\documentclass[11pt,reqno]{amsart}

\usepackage[T1]{fontenc}
\usepackage{lmodern}
\usepackage[utf8]{inputenc}
\usepackage{amsmath,amssymb,amsthm}
\usepackage{mathtools}
\usepackage{booktabs}
\usepackage{longtable}
\usepackage[margin=1.15in]{geometry}
\usepackage[colorlinks=true,linkcolor=blue,citecolor=blue,urlcolor=blue]{hyperref}

\usepackage[backend=biber,style=numeric,sorting=nyt,giveninits=true,
            maxbibnames=99,doi=true,isbn=false,url=false]{biblatex}
\DeclareSourcemap{
  \maps[datatype=bibtex]{
    \map{
      \step[fieldsource=mrnumber,fieldset=usera,origfieldval]
      \step[fieldset=mrnumber,null]
      \step[fieldset=mrclass,null]
      \step[fieldset=mrreviewer,null]
      \step[fieldset=fjournal,null]
    }
  }
}

\theoremstyle{plain}
\newtheorem{theorem}{Theorem}[section]
\newtheorem{proposition}[theorem]{Proposition}
\newtheorem{lemma}[theorem]{Lemma}
\newtheorem{corollary}[theorem]{Corollary}
\theoremstyle{definition}
\newtheorem{definition}[theorem]{Definition}
\newtheorem{remark}[theorem]{Remark}

\newtheorem{result}[theorem]{Result}

\DeclareMathOperator{\Tr}{Tr}

\newcommand{\F}{\mathbb{F}}
\newcommand{\N}{\mathbb{N}}
\newcommand{\Z}{\mathbb{Z}}
\newcommand{\R}{\mathbb{R}}
\newcommand{\C}{\mathbb{C}}
\newcommand{\Kx}{K^{\times}}

\newcommand{\muthree}{\mu_{3}}
\newcommand{\ReX}{}

\begin{document}

\title[Carlet's Kasami cyclic-additive conjecture]
      {Carlet's cyclic-additive conjecture for the Kasami monomials}

\author{G\'abor~P.~Nagy}
\email{nagyg@math.u-szeged.hu}
\address{Bolyai Institute, University of Szeged, Aradi V\'{e}rtan\'{u}k tere 1, H-6720 Szeged, Hungary}
\author{Douglas~S.~McNeil}
\email{dsm054@gmail.com}
\address{Unaffiliated}
\author{Attila~Vajda}
\email{at.vajda@gmail.com}
\address{Bolyai Institute, University of Szeged, Aradi V\'{e}rtan\'{u}k tere 1, H-6720 Szeged, Hungary}
\date{}

\begin{abstract}
Let $K$ be a finite field of characteristic two with $|K| = 2^{n}$, let
$\gcd(k,n) = 1$, let $d_{k} = 4^{k} - 2^{k} + 1$ be the Kasami exponent, and let
$\Delta_{k} = \{(b+1)^{d_{k}} + b^{d_{k}} + 1 : b \in K\}$ be the image of the
normalised derivative of the Kasami monomial in the direction $1$.  We show
that, for all distinct nonzero $v_{1},v_{2} \in K$,
\[
  \bigl|\{(x,y,z) \in \Delta_{k}^{3} :
     v_{1}x + v_{2}y + (v_{1}+v_{2})z = 0\}\bigr| = 2^{2n-3}.
\]
This establishes the cyclic-additive difference-set condition introduced by
Carlet and later posed for the Kasami functions at NSUCRYPTO~2019.  Starting
from the known half-size property of the derivative image, we express the
Fourier correction as twisted root counts and prove their required
nonnegativity by an incidence argument on the Fermat cubic.  An exact average
over the slopes then forces equality pointwise.  The argument covers every
admissible pair $(n,k)$ and has been formalised and machine-checked in Lean~4
with Mathlib.
\end{abstract}

\maketitle

\tableofcontents

\section{Introduction}

\subsection{Cyclic-additive difference sets}

Let $n \ge 2$ and let $K$ be a finite field with $|K| = Q = 2^{n}$. A \emph{cyclic difference set with Singer-like parameters} is a subset $\Delta \subseteq \Kx$ of size $2^{n-1}$ such that for all $v_{1}, v_{2} \in K$ with $v_{1}, v_{2} \neq 0$ and $v_{1} \neq v_{2}$,
\[
  \bigl|\{(x,y) \in \Delta^{2} : v_{1}x + v_{2}y  = 0\}\bigr| = 2^{n-2}.
\]
In general, showing that a subset $\Delta \subseteq K$ satisfies this property can be quite challenging; it often requires deep results from finite field theory, see \cite{DillonDobbertin}. Carlet generalized this notion to the componentwise setting (\cite[Definition~4.14]{Carlet2018}). A set $\Delta \subseteq K$ is a \emph{cyclic-additive difference set with Singer-like parameters} if
\begin{equation}\label{eq:carlet}
  \bigl|\{(x,y,z) \in \Delta^{3} :
    v_{1}x + v_{2}y + (v_{1}+v_{2})z = 0\}\bigr| = 2^{2n-3}.
\end{equation}
for all $v_{1}, v_{2} \in K$ with $v_{1}, v_{2} \neq 0$ and $v_{1} \neq v_{2}$. The value $2^{2n-3} = Q^{2}/8$ is what a random half of $K$ would give: the equation $v_{1}x + v_{2}y + (v_{1}+v_{2})z = 0$ is one linear condition on $\Delta^{3}$, of expected weight $|\Delta|^{3}/Q = Q^{2}/8$.  

Carlet introduced componentwise Walsh uniformity (CWU) to study the 
relationship between Walsh spectra and APN functions. He proved that every
quadratic APN function is CWU (\cite[Proposition~4.6]{Carlet2018}) and that
CWU implies APN-ness (\cite[Proposition~4.5]{Carlet2018}). Moreover, for a
power APN permutation $F$, CWU is equivalent to the condition that
\[
  \Delta_F = \{ F(x+1) + F(x) + 1 : x \in K \}
\]
is a cyclic-additive difference set with Singer-like parameters
(\cite[Proposition~4.10]{Carlet2018}).


\subsection{The conjecture}
  
The case of primary interest is the Kasami family of functions $F:K\to K$.  Fix $k$ with
$\gcd(k,n)=1$ and let
\[
  d_{k} = 4^{k} - 2^{k} + 1 = 2^{2k} - 2^{k} + 1
\]
be the \emph{Kasami exponent}, $F(x) = x^{d_{k}}$ the \emph{Kasami monomial.}  The
normalised derivative in the direction $1$ is
\[
  \delta_{k}(b) = (b+1)^{d_{k}} + b^{d_{k}} + 1 ,
\]
and $\Delta_{k} = \delta_{k}(K)$ is its image. Carlet verified \eqref{eq:carlet} computationally for all $1\leq k < n \le 11$ with $\gcd(n,k)=1$. He remarked that his proof for odd $n$ was ``not yet complete'' (\cite{cryptoeprint:2017/528}) and that even $n$ required a different method. He also asked for a direct proof (Conclusions of \cite{Carlet2018}). The problem was posed again at the Sixth International Olympiad in Cryptography (NSUCRYPTO 2019)~\cite{NSUCRYPTO}.
Nagy and Vajda~\cite{NagyVajda} proved \eqref{eq:carlet} for $k \bmod n \in \{1,2,n-2,n-1\}$ and verified it exhaustively for $n \le 13$.  Condition~\eqref{eq:carlet} also has a combinatorial reading, relating constancy of the three-variable solution count to constancy of the corresponding cubic Walsh correlation and to affine $3$-designs; see Tang~\cite{Tang}.  

\subsection{The result}

\begin{theorem}[Main theorem]\label{thm:main}
Let $K$ be a finite field of characteristic two with $|K| = 2^{n}$, let
$k \in \N$ satisfy $\gcd(k,n) = 1$, and let
$\Delta_{k} = \{\delta_{k}(b) : b \in K\}$ with
$\delta_{k}(b) = (b+1)^{d_{k}} + b^{d_{k}} + 1$ and
$d_{k} = 4^{k}-2^{k}+1$.  Then for all $v_{1},v_{2} \in K$ with
$v_{1} \neq 0$, $v_{2} \neq 0$, and $v_{1} \neq v_{2}$,
\[
  N_{k}(v_{1},v_{2}) :=
  \bigl|\{(x,y,z) \in \Delta_{k}^{3} :
     v_{1}x+v_{2}y+(v_{1}+v_{2})z=0\}\bigr| = 2^{2n-3}.
\]
\end{theorem}

No lower bound on $n$ and no range restriction on $k$ are assumed.  Both are
consequences of the remaining hypotheses: two distinct nonzero coefficients
force $|K| \ge 3$, hence $n \ge 2$, and coprimality then makes $k \bmod n$ a
positive representative below $n$; Section~\ref{sec:parameters} shows that
$\Delta_{k}$ and $N_{k}$ are unchanged when $k$ is replaced by $k \bmod n$.

Theorem~\ref{thm:main} has been formalised and machine-checked in Lean~4 on
top of Mathlib; see the Palomar entry~\cite{palomar-2026-09-04-000005-v1}.

\subsection{Outline of the proof}

Write $Q = |K| = 2^{n}$.  The argument has two halves, an analytic one
(Steps~2--3) and a geometric one (Step~4), linked by the half-size property
$2|\Delta_{k}| = Q$ of Section~\ref{sec:defs}.

\emph{Step 1 (Sections~\ref{sec:defs}--\ref{sec:parameters}): set up the
statement and reduce the parameter.}  Section~\ref{sec:defs} fixes the
notation, recasts the target count in slope form, and supplies the half-size
property of the Kasami derivative image.  Section~\ref{sec:parameters} shows
that reducing $k$ modulo $n$ leaves the image and count unchanged, and relates
the complementary parameters $k$ and $n-k$ by Frobenius.  Thus we may work
with $1 \le k < n$, while retaining the Frobenius transport needed later.

\emph{Step 2 (Section~\ref{sec:counting}): two elementary consequences of
half-size.}  Dividing the coefficient equation by $v_{1}$ turns
Theorem~\ref{thm:main} into a statement about the \emph{slope counts}
\[
  N(\rho) = \bigl|\{(x,y,z) \in \Delta^{3}: x + \rho y + (1+\rho) z = 0\}\bigr|,
  \qquad \rho \notin \{0,1\}.
\]
Additive-character orthogonality gives, for any $\Delta$ with
$2|\Delta| = Q$,
\[
  N(\rho) = \frac{Q^{2}}{8} + \frac{\ReX Z(\rho)}{Q},
  \qquad
  Z(\rho) = \sum_{\lambda \in \Kx}
     S(\lambda)\,S(\rho\lambda)\,S\bigl((1+\rho)\lambda\bigr),
\]
where $S=\widehat{\mathbf{1}_{\Delta}}$ is the Walsh transform of $\Delta$. A direct double count over
the pairs (slope, triple) gives the \emph{exact} average
\[
  \frac{1}{|\mathcal S|}\sum_{\rho \in \mathcal S} N(\rho) = \frac{Q^{2}}{8},
  \qquad \mathcal S = K \setminus \{0,1\}.
\]

\emph{Step 3 (Sections~\ref{sec:normalization}--\ref{sec:rootcount}): the
phase correction.}  After transporting the problem by Frobenius to an odd
parameter $k_{0}$ with $2^{k_{0}}+1 = 3m$ and $\gcd(m,2^{n}-1)=1$, the
half-size identity together with Gauss-sum algebra expresses the Walsh transform $S$ through Gauss sums. These turn the analytic quantity
$Z(\rho)$ into a count of solutions of an algebraic equation.  Writing the two
slopes as cubes, $\rho = A^{3}$ and $1+\rho = B^{3}$, and letting $D$ be an
inverse of $m$ modulo $2^{n}-1$, Theorem~\ref{thm:phaseroot} (\emph{phase-to-root-count identity}) states
\[
  Z(A^{3})
   = \frac{Q^{2}}{8}\Bigl(\sum_{u,v \in \muthree} R_{u,v}(A,B) - c^{2}\Bigr),
  \qquad
  R_{u,v}(A,B)
   = \bigl|\{t \in K: u\,t^{D} + v\,(At+B)^{D} = 1\}\bigr| ,
\]
where $\muthree \subseteq \Kx$ is the group of cube roots of unity and
$c = |\muthree| \in \{1,3\}$.  Its engine is Lemma~\ref{lem:weil}
(\emph{Weil sum evaluation}).  
In particular $Z(\rho)$ is a nonnegative real number as soon as each of the
$c^{2}$ root counts $R_{u,v}(A,B)$ is at least $1$, and a short rescaling
shows that this is guaranteed by one single arithmetic condition on the
exponent $m$, which we call $\mathrm{RootEqSolvable}(m,K)$:
\[
  \text{for all } p,q \in \Kx \text{ with } p^{3}+q^{3}=1
  \text{ there exist } w,z \in K \text{ with } w+z = 1
  \text{ and } w^{m}+p\,z^{m} = q .
\]
This condition is the only geometric input of the whole proof.

\emph{Step 4 (Section~\ref{sec:geometry}): geometry produces the roots.}  The
root equation is an incidence statement on the Fermat cubic
$E : W^{3}+T^{3}=1$ in characteristic two, for the map
$\Phi_{k}(\mathcal{Q}) = \boxminus(\mathcal{Q} \boxplus \pi^{k}\mathcal{Q}) \boxplus T_{3}$,
where $\pi$ is the Frobenius
endomorphism and $T_{3}$ a distinguished $3$-torsion point.  For odd $n$ the
rational kernel of $1+\pi^{k}$ is trivial, so $\Phi_{k}$ is bijective on the
finite group $E(K)$.  For even $n$ we factor $1+\pi^{k} = (1+\pi)G$ inside
$\Z[\pi] \cong \Z[\sqrt{-2}]$, invert the prime-to-three factor $G$ over $K$,
and lift through the degree-three Hessian quotient using explicit coordinates.
Either way every twisted fibre contains a root, whence $\ReX Z(\rho) \ge 0$ and
\[
  N(\rho) \ge \frac{Q^{2}}{8} \qquad \text{for every admissible slope.}
\]

\emph{Step 5 (Section~\ref{sec:assembly}): rigidity.}  Every slope count is at
least $Q^{2}/8$ and their average is exactly $Q^{2}/8$; hence all of them equal
$Q^{2}/8 = 2^{2n-3}$.  Transporting back to the original coefficients and to
the original parameter $k$ finishes the proof.

\subsection{Relation to earlier work}

The immediate antecedent is Nagy--Vajda~\cite{NagyVajda}.  The Frobenius
transport between complementary parameters, the reduction of the
character-sum correction to root counts, the strategy of a pointwise lower
bound combined with an exact average, and the reading of the root equation on
the Fermat cubic (their Remark~11.13, for $k=2$) all originate there; the
present proof extends the last of these to every admissible $k$.

The Kasami exponent and its spectral theory go
back to Kasami~\cite{Kasami} and Dobbertin~\cite{DobbertinKasami}; the
exceptional polynomials $M_{k}$ to Müller~\cite{Muller} and
Cohen--Matthews~\cite{CohenMatthews}, with the connection to the Kasami
derivative in Dobbertin~\cite{DobbertinNATO} and
Dillon--Dobbertin~\cite{DillonDobbertin}; the phase formula to Dillon and
Kashyap~\cite{DillonKashyap}; the moment-to-root-count conversion to the Weil
sum tradition of Helleseth~\cite{Helleseth} and Katz~\cite{Katz}; and the
Hessian addition law to Cauchy and Sylvester, see
Chudnovsky--Chudnovsky~\cite{Chudnovsky} and, for the degree-three Hessian
isogeny, \cite{TwistedHessian,HessianIsogenies}.

\section{Notation, definitions, and preliminary results}\label{sec:defs}

Throughout, $K$ is a finite field of characteristic two with
\[
  Q = |K| = 2^{n}, \qquad N = |\Kx| = Q-1 .
\]
All characters take values in $\C$.  We fix once and for all a primitive
additive character
\[
  \psi \colon (K,+) \to \C^{\times},
\]
concretely $\psi(x) = (-1)^{\Tr_{K/\F_{2}}(x)}$.  Since
$\operatorname{char} K = 2$ we have $\psi(x)^{2} = \psi(2x) = 1$, so $\psi$
takes only the values $\pm 1$, and $\psi(x^{2}) = \psi(x)$ because the trace is
Frobenius invariant.  For a multiplicative character
$\chi \colon \Kx \to \C^{\times}$, extended by $\chi(0)=0$, the Gauss sum is
\[
  G(\chi) = \sum_{x \in \Kx} \chi(x)\psi(x).
\]

\begin{definition}[The four statement-level objects]\label{def:objects}
For $k \in \N$ put
\[
  d_{k} = 4^{k} - 2^{k} + 1 \in \N,
  \qquad
  \delta_{k}(b) = (b+1)^{d_{k}} + b^{d_{k}} + 1 \quad (b \in K),
\]
\[
  \Delta_{k} = \delta_{k}(K) \subseteq K,
  \qquad
  N_{k}(v_{1},v_{2}) =
    \bigl|\{(x,y,z) \in \Delta_{k}^{3} :
        v_{1}x+v_{2}y+(v_{1}+v_{2})z = 0\}\bigr| .
\]
\end{definition}

\begin{definition}[Slopes]\label{def:slopes}
For $\rho \in K$ set
\[
  N(\rho) = \bigl|\{(x,y,z) \in \Delta_{k}^{3} :
     x + \rho y + (1+\rho)z = 0\}\bigr|.
\]
A slope $\rho$ is \emph{admissible} if
$\rho \neq 0$ and $\rho \neq 1$; we write
$\mathcal S = \{ \rho \in K : \rho \neq 0,\ \rho \neq 1\}$.  The \emph{Walsh coefficient} of
$\Delta \subseteq K$ is
\[
  S(a) = \widehat{\mathbf{1}_{\Delta}}(a) = \sum_{x \in \Delta} \psi(ax) \qquad (a \in K),
\]
and the \emph{phase triple sum} at a pair
$(\rho,\sigma)$ is
\[
  Z(\rho) = \sum_{\lambda \in \Kx}
     S(\lambda) S(\rho\lambda) S(\sigma\lambda),
  \qquad \sigma = 1+\rho.
\]
\end{definition}

The passage between the coefficient form and the slope form is pure
bookkeeping.

\begin{lemma}[Coefficient form versus slope form]\label{lem:coefslope}
If $v_{1} \neq 0$ then $N_{k}(v_{1},v_{2}) = N(v_{2}/v_{1})$.  If moreover
$v_{2} \neq 0$ and $v_{1} \neq v_{2}$, then $\rho = v_{2}/v_{1}$ is admissible.
Conversely, if $N(\rho) = c$ for every admissible $\rho$, then
$N_{k}(v_{1},v_{2}) = c$ for all admissible coefficient pairs.
\end{lemma}
\begin{proof}
Dividing $v_{1}x+v_{2}y+(v_{1}+v_{2})z = 0$ by $v_{1}$ gives
$x + \rho y + (1+\rho)z = 0$ with $\rho = v_{2}/v_{1}$, and multiplication by
$v_{1} \neq 0$ is a bijection of the solution set of one equation onto the
other; the two filtered subsets of $\Delta_{k}^{3}$ are literally equal.  If
$v_{2} \neq 0$ then $\rho \neq 0$, and $v_{1} \neq v_{2}$ gives $\rho \neq 1$.
\end{proof}

We record the arithmetic identity that turns the analytic main term
into the integer of Theorem~\ref{thm:main}: if $|K| = 2^{n}$ with $n \ge 2$
then
\begin{equation}\label{eq:mainterm}
  \frac{Q^{2}}{8} = 2^{2n-3};
\end{equation}
note $2n-3 \ge 1$ exactly because
$n\ge 2$.

With the statement-level objects in place, the first Kasami-specific input
is the APN property of the Kasami monomial. In direction $1$, its derivative
has fibres of size two: any preimage $b$ is paired with $b+1$. Since
$\delta_k(b)$ is that derivative plus $1$, it has the same fibres. Thus its
image has half the field's elements:

\begin{result}[Half-size]\label{res:halfsize}
Let $|K| = 2^{n}$, $1 \le k < n$, $\gcd(k,n) = 1$.  Then
\[
  2\,|\Delta_{k}| = |K| .
\]
\end{result}
\begin{proof}
This is the half-size consequence of \cite[Theorem A]{DillonDobbertin}; see
\cite[Corollary 2]{DobbertinNATO} for a direct proof of APN-ness.
\end{proof}

We also record, for later use, the elementary fact
\begin{equation}\label{eq:zeroinDelta}
  0 = \delta_{k}(0) \in \Delta_{k}.
\end{equation}

\section{Parameter reduction and Frobenius transport}\label{sec:parameters}

\subsection{Reduction of the parameter modulo the degree}\label{subsec:reduction}

The literature statement imposes only $\gcd(k,n)=1$.  The substantive proof is
carried out for $1 \le k < n$.  

\begin{lemma}\label{lem:expperiod}
For all $n,k$ one has $d_{k} \equiv d_{k \bmod n} \pmod{2^{n}-1}$.
\end{lemma}
\begin{proof}
Write $k = (k \bmod n) + n\lfloor k/n\rfloor$.  Since $2^{n} \equiv 1
\pmod{2^{n}-1}$ we get $2^{k} \equiv 2^{k \bmod n}$ and, squaring,
$4^{k} \equiv 4^{k\bmod n} \pmod{2^{n}-1}$.  As $2^{j} \le 4^{j}$ for all $j$,
the truncated differences also agree modulo $2^{n}-1$, and adding $1$
preserves the congruence.
\end{proof}

\begin{lemma}\label{lem:powperiod}
Let $x \in K$ and let $a,b \ge 1$ with $a \equiv b \pmod{|K|-1}$.  Then
$x^{a} = x^{b}$.
\end{lemma}
\begin{proof}
For $x = 0$ both sides vanish because $a,b \ge 1$.  For $x \neq 0$ use
$x^{|K|-1} = 1$ and split each exponent as remainder plus multiple of $|K|-1$.
\end{proof}

\begin{proposition}\label{prop:reduction}
Let $|K| = 2^{n}$.  Then for every $k$,
\[
  \delta_{k} = \delta_{k \bmod n}, \qquad
  \Delta_{k} = \Delta_{k \bmod n}, \qquad
  N_{k}(v_{1},v_{2}) = N_{k \bmod n}(v_{1},v_{2}).
\]
Moreover, if $|K|=2^{n}$ and $K$ contains two distinct nonzero elements
$v_{1} \neq v_{2}$, then $n \ge 2$; and if $n \ge 2$ and $\gcd(k,n)=1$ then
$r = k \bmod n$ satisfies $1 \le r < n$ and $\gcd(r,n)=1$.
\end{proposition}
\begin{proof}
$d_{k} \ge 1$ always, so Lemmas~\ref{lem:expperiod} and~\ref{lem:powperiod}
give $x^{d_{k}} = x^{d_{k \bmod n}}$ for all $x \in K$; applying this to
$x = b+1$ and $x = b$ gives $\delta_{k} = \delta_{k \bmod n}$, and the image and
the count follow.  For the second part, $0, v_{1}, v_{2}$ are three distinct
elements, so $2^{n} = |K| \ge 3$, forcing $n \ge 2$.  Finally $\gcd(k,n)=1$
gives $\gcd(k \bmod n, n) = 1$, and $k \bmod n = 0$ would force $n = 1$.
\end{proof}

From now on we may and do assume
\[
  1 \le k < n, \qquad \gcd(k,n) = 1, \qquad n \ge 2 .
\]

\subsection{The Frobenius transport}

The following lemma shows that the derivative image $\Delta_{k}$ is closely related to $\Delta_{n-k}$ via a Frobenius automorphism. This transport will let us choose an odd parameter in the normalization step below.

\begin{lemma}[Frobenius transport]\label{lem:complement}
Let $|K| = 2^{n}$ and $k \le n$.  Then for all $b \in K$,
\[
  \delta_{k}(b) = \bigl(\delta_{n-k}(b)\bigr)^{2^{2k}},
\]
so $\Delta_{k}$ is the image of $\Delta_{n-k}$ under the automorphism
$x \mapsto x^{2^{2k}}$; in particular $|\Delta_{k}| = |\Delta_{n-k}|$.
\end{lemma}
\begin{proof}
A direct computation with the Kasami exponents gives the integer identity
\[
  2^{2k}\,d_{n-k} = d_{k} + (2^{n}-1)\bigl((2^{n}-2^{k})+1\bigr),
\]
and exponents may be shifted by multiples of $|K|-1 = 2^{n}-1$ without changing
the value of $x \mapsto x^{e}$ for $e \ge 1$.  Hence
$x^{d_{n-k}2^{2k}} = x^{d_{k}}$ for all $x \in K$, and applying this to $b+1$
and $b$ and using that $x \mapsto x^{2^{2k}}$ is additive gives the claim.
\end{proof}

\section{Counting identities from half-size alone}\label{sec:counting}

In this section $\Delta \subseteq K$ is an \emph{arbitrary} subset of a
characteristic-two finite field with
\[
  2|\Delta| = Q = |K| ,
\]
and $S(a)$, $N(\rho)$, $Z(\rho)$ are as in Definition~\ref{def:slopes}, with
$\Delta$ in place of $\Delta_{k}$.  Nothing Kasami-specific is used.

\begin{theorem}[Walsh triple-count formula]\label{thm:walshcount}
For every $\rho \in K$,
\[
  N(\rho) = \frac{Q^{2}}{8} + \frac{\ReX Z(\rho)}{Q}.
\]
\end{theorem}
\begin{proof}
Orthogonality of $\psi$ gives, for $t \in K$,
$\sum_{a \in K}\psi(at) = Q\cdot \mathbf 1_{t=0}$.  Hence
\[
  Q\, N(\rho)
  = \sum_{(x,y,z) \in \Delta^{3}} \ \sum_{a \in K}
      \psi\bigl(a(x+\rho y+(1+\rho)z)\bigr)
  = \sum_{a \in K} S(a)\,S(\rho a)\,S\bigl((1+\rho)a\bigr),
\]
because for fixed $a$ the character factorises over the three coordinates.
The term $a=0$ contributes $S(0)^{3} = |\Delta|^{3} = Q^{3}/8$, and the
remaining terms are by definition $Z(\rho)$.  Dividing by $Q$ gives the formula.
\end{proof}

\begin{theorem}[Exact slope average]\label{thm:average}
Assume $Q > 2$ and $2|\Delta| = Q$.  Then
\[
  \frac{1}{|\mathcal S|}\sum_{\rho \in \mathcal S} N(\rho) = \frac{Q^{2}}{8},
  \qquad \mathcal S = K \setminus\{0,1\},\quad |\mathcal S| = Q-2 .
\]
\end{theorem}
\begin{proof}
We double count the set of pairs $(\rho, (x,y,z))$ with $\rho \in \mathcal S$,
$(x,y,z) \in \Delta^{3}$ and $x+\rho y+(1+\rho)z = 0$.  Fix a triple
$(x,y,z)$.

If $y = z$ the equation reads $x+y = 0$, so it holds for every admissible
$\rho$ exactly when $x=y=z$.  If $y \neq z$, the unique solution is
$\rho = (x+z)/(y+z)$, and this is admissible exactly when $x,y,z$ are pairwise
distinct.  Hence
\[
  \sum_{\rho \in \mathcal S} N(\rho)
   = (Q-2)|\Delta| + |\Delta|(|\Delta|-1)(|\Delta|-2).
\]
Substituting $|\Delta|=Q/2$ and dividing by $Q-2$ gives
\[
  \frac{1}{Q-2}\sum_{\rho\in\mathcal S}N(\rho)
  = \frac{Q}{2} + \frac{Q}{2}\left(\frac{Q}{4}-1\right)
  = \frac{Q^{2}}{8}.
\]
\end{proof}

Theorem~\ref{thm:average} is the ``conservation law'' of the argument.  Note
that it holds for any half-size set; what is special about $\Delta_{k}$ is the
pointwise lower bound proved in the next four sections.

\section{Normalisation of the parameter and of the coefficients}\label{sec:normalization}

The analytic part of the argument needs an odd parameter for which
$2^{k}+1$ factors as $3m$ with $m$ prime to $2^{n}-1$.  Both conditions can be
arranged inside the pair $\{k, n-k\}$.

\begin{proposition}\label{prop:normparam}
Let $k<n$ with $\gcd(k,n)=1$.  Then there are $k_{0}, m \in \N$ with
\[
  k_{0} \in \{k,\, n-k\}, \qquad k_{0} \text{ odd}, \qquad \gcd(k_{0},n)=1,
  \qquad 2^{k_{0}}+1 = 3m, \qquad \gcd(m, 2^{n}-1) = 1 .
\]
\end{proposition}
\begin{proof}
We first show that an odd $k_{0}$ coprime to $n$ works as soon as either $n$ is
odd or $k_{0}\not\equiv 3 \pmod 6$.  For odd
$k_{0}$ one has $2^{k_{0}}\equiv 2 \pmod 3$, so $3 \mid 2^{k_{0}}+1$ and
$2^{k_{0}}+1 = 3m$ for some $m$.  Since $\gcd(k_{0},n)=1$, we have
$\gcd(2^{2k_{0}}-1, 2^{n}-1) = 2^{\gcd(2k_{0},n)}-1 \in \{1,3\}$, and
$2^{k_{0}}+1 \mid 2^{2k_{0}}-1$, so every common divisor
of $2^{k_{0}}+1$ and $2^{n}-1$ divides $3$; hence $\gcd(m,2^{n}-1)$ divides $3$
as well, and equals $3$ only if $9 \mid 3m = 2^{k_{0}}+1$ and $3 \mid 2^{n}-1$.
Now $9 \mid 2^{k_{0}}+1$ holds exactly when $k_{0}\equiv 3 \pmod 6$, and $3 \mid 2^{n}-1$ exactly when
$n$ is even.  So under either of the two stated conditions
$\gcd(m,2^{n}-1)=1$.

It remains to choose $k_{0}\in\{k,n-k\}$ accordingly; in both cases
$1 \le k_{0} < n$ and $\gcd(k_{0},n) = \gcd(k,n) = 1$.  If $n$ is odd, exactly
one of $k$, $n-k$ is odd, and we take that one; no residue condition is needed.
If $n$ is even, then $k$ and $n-k$ are both odd; if $k \not\equiv 3 \pmod 6$ we
take $k_{0}=k$, and otherwise $3 \mid k$, so $3 \nmid n$ by coprimality, hence
$n \equiv 2$ or $4 \pmod 6$ and $n-k \equiv 5$ or $1 \pmod 6$; in particular
$n-k \not\equiv 3 \pmod 6$ and we take $k_{0}=n-k$.
\end{proof}

\begin{proposition}[Coefficient transport]\label{prop:transport}
Let $|K|=2^{n}$, $k \le n$, and let $v_{1},v_{2}$ be distinct nonzero
elements.  Then there are distinct nonzero $w_{1},w_{2}$ with
\[
  N_{k}(v_{1},v_{2}) = N_{n-k}(w_{1},w_{2}).
\]
\end{proposition}
\begin{proof}
By Lemma~\ref{lem:complement}, $\Delta_{k} = \phi(\Delta_{n-k})$ where
$\phi(x) = x^{2^{2k}}$ is a field automorphism.  Applying $\phi^{-1}$ to a
solution triple of $v_{1}x+v_{2}y+(v_{1}+v_{2})z=0$ in $\Delta_{k}^{3}$ gives a
solution triple in $\Delta_{n-k}^{3}$ of the equation with coefficients
$w_{i} = \phi^{-1}(v_{i})$, and this is a bijection between the two solution
sets.  Distinctness and nonvanishing are preserved by $\phi^{-1}$.
\end{proof}

\begin{remark}
It is important that the transport moves the \emph{coefficients}: it is
\emph{not} true that $N(\rho)$ at parameter $k$ equals $N(\rho)$ at parameter
$n-k$ for the same slope $\rho$.  Frobenius acts on the slope too.
\end{remark}

Combining Propositions~\ref{prop:normparam} and~\ref{prop:transport} with
Result~\ref{res:halfsize}, we may and do assume from now on, for the purpose
of proving the pointwise lower bound, that
\begin{equation}\label{eq:normalized}
  1 \le k < n, \quad k \text{ odd}, \quad \gcd(k,n)=1, \quad
  2^{k}+1 = 3m, \quad \gcd(m,2^{n}-1)=1, \quad 2|\Delta_{k}| = Q .
\end{equation}

\section{The multiplicative phase formula}\label{sec:phase}

Define the \emph{sign function} of $\Delta = \Delta_{k}$ on $\Kx$ by
\[
  F(x) = \begin{cases} +1, & x \in \Delta, \\ -1, & x \notin \Delta,\end{cases}
\]
and its multiplicative transform
$\widehat F(\chi) = \sum_{x \in \Kx} F(x)\chi(x)$.

\begin{theorem}[Phase formula]\label{thm:phase}
Assume \eqref{eq:normalized}.  Then for every multiplicative character $\chi$
of $K$,
\[
  \widehat F(\chi)
   = \frac{G(\chi)\,G\bigl(\chi^{2^{k}+1}\bigr)}{G\bigl(\chi^{3}\bigr)} .
\]
\end{theorem}
\begin{proof}
This is the formula of Dillon and Kashyap~\cite[Theorems 1--2]{DillonKashyap},
after changing their sign convention because their difference set is the
complement of $\Delta$.  The displayed expression is well defined:
$G(\chi^{3}) \neq 0$ for every $\chi$ because $\psi$ is primitive; see Lemma \ref{lem:gauss-sum}.
\end{proof}

We record the Gauss-sum identities needed to derive the all-character Walsh formula below.

\begin{lemma}\label{lem:gauss-sum}
For a multiplicative character $\chi: \Kx \to \C^{\times}$, let
\[ G(\chi) = \sum_{x \in \Kx} \chi(x) \psi(x) \]
be its Gauss sum. Then $G(1)=-1$ and $G(\chi)G(\chi^{-1}) = Q$ for nontrivial $\chi$. Moreover, the multiplicative Fourier inversion formula
\[ \sum_\chi G(\chi^{-1}) \chi(x) = (Q-1) \psi(x)\]
holds for all $x \in \Kx$.
\end{lemma}
\begin{proof}
  For the trivial character $1$, we have $G(1) = \sum_{x \in \Kx} \psi(x) = -1$ because $\sum_{x \in K} \psi(x) = 0$ and $\psi(0) = 1$. Fix a nontrivial character $\chi$. Then
  \begin{align*}
    G(\chi)G(\chi^{-1}) &= \sum_{x \in \Kx} \sum_{y \in \Kx} \chi(x) \chi^{-1}(y) \psi(x) \psi(y) \\
    & = \sum_{x,y \in \Kx} \chi(x/y) \psi(x+y) \\
    & = \sum_{y,u \in \Kx} \chi(u) \psi(y(u+1)) \quad \text{with } u = x/y \\
    & = \chi(1) \sum_{y \in \Kx} \psi(0) + \sum_{u \neq 0,1} \chi(u) \sum_{y\in \Kx} \psi(y(u+1)) \\
    & = 1\cdot (2^n-1) - \sum_{u \neq 0,1} \chi(u)  \\
    & = 2^n -\sum_{u \in \Kx} \chi(u) = 2^n = Q.
  \end{align*}
  For the multiplicative Fourier inversion formula, 
  \begin{align*}
    \sum_\chi G(\chi^{-1}) \chi(x) &= \sum_\chi \sum_{y \in \Kx} \chi^{-1}(y) \psi(y) \chi(x) \\
    &=\sum_{y \in \Kx} \psi(y) \sum_\chi \chi(x/y) \\
    &= (Q-1) \psi(x),
  \end{align*}
  where we used the orthogonality of characters: $\sum_\chi \chi(x) = 2^n-1$ if $x=1$ and $0$ otherwise.
\end{proof}

The phase formula is the input for the Fourier analysis; we restate it in the
form actually consumed downstream.

\begin{theorem}[All-character Walsh formula]\label{thm:walshchar}
Assume \eqref{eq:normalized} and put $e = 2^{k}+1$.  Then for every
$a \in \Kx$,
\[
  2\,S(a) = \frac{Q}{N}\sum_{\chi}
      \frac{G(\chi^{e})}{G(\chi^{3})}\,\chi(a),
\]
the sum ranging over all $N = |\Kx|$ multiplicative characters of $K$.
\end{theorem}
\begin{proof}
On $\Kx$ we have $1_{\Delta} = (1+F)/2$, so by \eqref{eq:zeroinDelta}
\[
  2S(a) = 2\sum_{x \in \Delta}\psi(ax)
        = 2 + \sum_{x \in \Kx}\bigl(1+F(x)\bigr)\psi(ax)
        = 1 + \sum_{x\in\Kx}F(x)\psi(ax),
\]
using $\sum_{x \in \Kx}\psi(ax) = -1$ for $a \neq 0$.  Multiplicative
Fourier inversion on $\Kx$ (Lemma \ref{lem:gauss-sum}) turns the last sum into
\[
  \sum_{x\in\Kx}F(x)\psi(ax)
   = \frac{1}{N}\sum_{\chi}\widehat F(\chi)\,\chi(a)\,G(\chi^{-1}),
\]
and inserting Theorem~\ref{thm:phase} together with
$G(\chi)G(\chi^{-1}) = Q$ for $\chi \neq 1$ and $G(1) = -1$ converts each term
into $\frac{Q}{N}\frac{G(\chi^{e})}{G(\chi^{3})}\chi(a)$, except for the
principal character, whose correction is exactly the constant $1$ subtracted
above.
\end{proof}

\section{From the phase to twisted root counts}\label{sec:rootcount}

The results of this section are more general and apply to any function $S$ satisfying the all-character Walsh formula (Theorem~\ref{thm:walshchar}), not just the specific function arising from the set $\Delta_k$. The parameters are $n$ and $m$, where $n$ is the degree of the field extension $K/\F_2$ and $\gcd(m,N) = 1$ with $N = |\Kx|$. Moreover, we define $e = 3m$ and choose $D$ such that $mD \equiv 1 \pmod{N}$. Then $x \mapsto x^{D}$ and $x \mapsto x^{m}$ are mutually inverse bijections of $K$.

\begin{definition}\label{def:rootcount}
Let $\muthree = \{u \in K : u^{3}=1\}$ and $c = |\muthree| \in \{1,3\}$; $c=3$
if $n$ is even and $c=1$ if $n$ is odd.  Put
\[
  \Phi(x) = \sum_{u \in \muthree}\psi\bigl(u x^{D}\bigr),
  \qquad
  \widehat\Phi(z) = \sum_{t \in K}\Phi(t)\psi(zt),
\]
\[
  W_{u,v}(A,B) = \sum_{x,y \in K}
      \psi\bigl(u x^{D} + v (Ax+By)^{D} + y^{D}\bigr),
\]
\[
  R_{u,v}(A,B) = \bigl|\{t \in K : u t^{D} + v (At+B)^{D} = 1\}\bigr| .
\]
\end{definition}

\begin{theorem}[Phase-to-root-count identity]\label{thm:phaseroot}
Let $\psi$ be primitive, $e = 3m$, $mD \equiv 1 \pmod N$, and let
$S \colon K \to \C$ satisfy the all-character Walsh formula of
Theorem~\ref{thm:walshchar}.  Let $A,B \in K$ with $A \neq 0$,
$A^{3}+B^{3} = 1$, and $A^{3}\neq 1$.  Then
\[
  Z(A^{3})
  = \frac{Q^{2}}{8}\Bigl(
      \sum_{u,v \in \muthree} R_{u,v}(A,B) - c^{2}\Bigr),
\]
where $Z$ is the phase triple sum at the pair $(\rho,\sigma) = (A^{3},B^{3})$.
\end{theorem}

The rest of this section is devoted to the proof of Theorem~\ref{thm:phaseroot}.
It is a chain of elementary identities: a closed form for $S$
(\S\ref{ssec:closedform}), the identification of $2S$ at cubes with the Fourier
transform $\widehat\Phi$ (\S\ref{ssec:transform}), a Parseval step and a
symmetrisation step (\S\ref{ssec:parseval}), and finally the evaluation of the
Weil sums as root counts (\S\ref{ssec:weil}).  None of the steps uses anything
about $\Delta_{k}$: the only input is the all-character Walsh formula, so the
theorem is a statement about an arbitrary function $S$ satisfying it.  We
assemble the pieces in \S\ref{ssec:assemblyroot}.

Throughout, $\psi$ is the fixed primitive additive character, $e = 3m$ with
$m \ge 1$, $D \ge 1$ satisfies $mD \equiv 1 \pmod N$, and $S \colon K \to \C$
satisfies
\begin{equation}\label{eq:walshhyp}
  2S(a) = \frac{Q}{N}\sum_{\chi}\frac{G(\chi^{3m})}{G(\chi^{3})}\chi(a)
  \qquad (a \in \Kx),
\end{equation}
the sum being over all $N$ multiplicative characters of $K$.  Recall the two
standing conventions: a multiplicative character is extended to $K$ by
$\chi(0) = 0$, and $\Kx$-indexed sums of a function vanishing at $0$ may be
extended to $K$ without change.

\subsection{Preliminaries on the exponent \texorpdfstring{$D$}{D} and on
\texorpdfstring{$\muthree$}{mu3}}\label{ssec:prelimroot}

\begin{lemma}[The power map $x \mapsto x^{D}$]\label{lem:powD}
Let $m,D \ge 1$ with $mD \equiv 1 \pmod{N}$.  Then:
\begin{enumerate}
\item $x^{mD} = x$ for every $x \in K$; consequently $(x^{D})^{m} = x$ and
  $(x^{m})^{D} = x$, so $x \mapsto x^{D}$ and $x \mapsto x^{m}$ are mutually
  inverse bijections of $K$ fixing $0$ and $1$;
\item for every $C \in K$,
  \[
    \sum_{y \in K}\psi\bigl(C y^{D}\bigr) =
      \begin{cases} Q, & C = 0,\\ 0, & C \neq 0.\end{cases}
  \]
\end{enumerate}
\end{lemma}
\begin{proof}
(1) Write $mD = 1 + Nj$ with $j \ge 0$, which is possible since
$mD \ge 1$ and $mD \equiv 1 \pmod N$.  For $x = 0$ the claim is
$0^{mD} = 0$, true because $mD \ge 1$.  For $x \in \Kx$ we have $x^{N} = 1$ by
Lagrange, so $x^{mD} = x\cdot (x^{N})^{j} = x$.  Since $mD = Dm$, both
$(x^{D})^{m} = x$ and $(x^{m})^{D} = x$ follow, which exhibits the two maps as
mutually inverse.

(2) By (1) the substitution $y = x^{m}$, i.e.\ $x = y^{D}$, is a bijection of
$K$, so $\sum_{y}\psi(Cy^{D}) = \sum_{x}\psi(Cx)$.  The latter is the sum of
the shifted character $\psi_{C} = \psi(C\,\cdot\,)$ over the additive group $K$;
it equals $Q$ if $\psi_{C}$ is trivial and $0$ otherwise, and by primitivity of
$\psi$ the character $\psi_{C}$ is trivial exactly for $C = 0$.
\end{proof}

\begin{lemma}[The group $\muthree$]\label{lem:mu3basic}
$\muthree = \{u \in K : u^{3} = 1\}$ is a subgroup of $\Kx$ of order
$c = \gcd(3,Q-1)$; thus $c = 3$ if $n$ is even and $c = 1$ if $n$ is odd, and
in particular $c \neq 0$.  Moreover:
\begin{enumerate}
\item for every $A \in \Kx$ the fibre of the cubing map through $A$ is
  $\{b \in \Kx : b^{3} = A^{3}\} = A\muthree$, of cardinality exactly $c$;
\item if $c = 1$ then every element of $\Kx$ is a cube;
\item $u \mapsto (u^{-1})^{D}$ is a bijection of $\muthree$ onto itself, with
  inverse $v \mapsto (v^{m})^{-1}$;
\item for $b \in \muthree$, both $a \mapsto ab^{-1}$ and $g \mapsto gb^{-1}$
  are bijections of $\muthree$ onto itself.
\end{enumerate}
\end{lemma}
\begin{proof}
$\muthree$ is the kernel of the endomorphism $x \mapsto x^{3}$ of the cyclic
group $\Kx$ of order $N = Q-1$, hence a subgroup of order $\gcd(3,N)$.  Since
$Q = 2^{n}$ and $2 \equiv -1 \pmod 3$, we get $3 \mid 2^{n}-1$ iff $n$ is even.
(1) is the coset description of a kernel fibre: $b^{3} = A^{3}$ iff
$(bA^{-1})^{3} = 1$ iff $b \in A\muthree$, and $b \mapsto bA^{-1}$ is a
bijection $A\muthree \to \muthree$.  (2) If $c=1$ the cubing map is injective
on the finite group $\Kx$, hence surjective.  (3) If $u^{3} = 1$ then
$((u^{-1})^{D})^{3} = ((u^{3})^{-1})^{D} = 1$, so the map sends $\muthree$ into
itself.  Writing $\varphi(u) = (u^{-1})^{D}$ and $\varphi'(v) = (v^{m})^{-1}$,
Lemma~\ref{lem:powD}(1) gives
$\varphi'(\varphi(u)) = \bigl(((u^{-1})^{D})^{m}\bigr)^{-1} = (u^{-1})^{-1} = u$
and
$\varphi(\varphi'(v)) = \bigl(((v^{m})^{-1})^{-1}\bigr)^{D} = (v^{m})^{D} = v$,
so the two maps are mutually inverse bijections of $\muthree$.  (4) is the
bijectivity of translation in a group.
\end{proof}

\subsection{The closed form of the phase}\label{ssec:closedform}

The first step removes all Gauss sums from \eqref{eq:walshhyp}.  Write
$\theta = \chi^{3}$, so that $\chi^{e} = \chi^{3m} = \theta^{m}$.

\begin{lemma}[Gauss-ratio decomposition]\label{lem:ratio}
For every multiplicative character $\chi$ of $K$, with $\theta = \chi^{3}$,
\[
  \frac{G(\chi^{3m})}{G(\chi^{3})}
   = \frac{G(\theta^{m})\,G(\theta^{-1})}{Q}
     + \begin{cases} 1 - \dfrac1Q, & \theta = 1,\\[2pt] 0, & \theta \neq 1 .\end{cases}
\]
\end{lemma}
\begin{proof}
Both sides are well defined because $G(\theta) \neq 0$ for every $\theta$.
If $\theta = 1$ then also
$\theta^{m} = 1$ and $\theta^{-1} = 1$, so the left-hand side is
$G(1)/G(1) = 1$ while the right-hand side is
$G(1)^{2}/Q + 1 - 1/Q = 1/Q + 1 - 1/Q = 1$, using $G(1) = -1$.  If
$\theta \neq 1$ then $G(\theta)G(\theta^{-1}) = Q$, i.e.\
$G(\theta^{-1})/Q = G(\theta)^{-1}$, and the right-hand side is
$G(\theta^{m})/G(\theta)$, which is the left-hand side.
\end{proof}

\begin{lemma}[Orthogonality for the Gauss products]\label{lem:gaussprod}
For every $a \in \Kx$,
\[
  \sum_{\chi} G\bigl((\chi^{3})^{m}\bigr)\,G\bigl((\chi^{3})^{-1}\bigr)\,\chi(a)
   = N \sum_{x \in \Kx}\ \sum_{\substack{y \in \Kx\\ y^{3} = a\,x^{3m}}}\psi(x+y).
\]
\end{lemma}
\begin{proof}
Expand both Gauss sums as sums over $\Kx$ (legitimate since $\chi(0)=0$):
\[
  G\bigl((\chi^{3})^{m}\bigr)G\bigl((\chi^{3})^{-1}\bigr)\chi(a)
   = \sum_{x \in \Kx}\sum_{y \in \Kx}
      \chi^{3m}(x)\,\chi^{-3}(y)\,\chi(a)\,\psi(x)\psi(y).
\]
Since $\chi$ is multiplicative and $\psi(x)\psi(y) = \psi(x+y)$, the summand
equals $\chi\bigl(x^{3m}y^{-3}a\bigr)\psi(x+y)$.  Summing over all $\chi$ and
interchanging the (finite) sums, the first orthogonality relation
\[
  \sum_{\chi}\chi(g) = \begin{cases} N, & g = 1,\\ 0, & g \neq 1\end{cases}
  \qquad (g \in \Kx)
\]
gives
\[
  \sum_{\chi}G\bigl((\chi^{3})^{m}\bigr)G\bigl((\chi^{3})^{-1}\bigr)\chi(a)
   = N \sum_{\substack{x,y \in \Kx\\ x^{3m}y^{-3}a = 1}}\psi(x+y).
\]
Finally $x^{3m}y^{-3}a = 1$ is equivalent to $y^{3} = a\,x^{3m}$, which is the
indexing of the asserted inner sum.
\end{proof}

\begin{lemma}[Closed form of the phase]\label{lem:closedform}
For every $a \in \Kx$,
\[
  2S(a) = \bigl|\{b \in \Kx : b^{3} = a\}\bigr|
        + \sum_{x \in \Kx}\ \sum_{\substack{y \in \Kx\\ y^{3} = a x^{3m}}}
          \psi(x+y) .
\]
\end{lemma}
\begin{proof}
Insert Lemma~\ref{lem:ratio} into \eqref{eq:walshhyp} and split the resulting
sum into the two contributions.  The first contributes
\[
  \frac{Q}{N}\cdot\frac1Q
    \sum_{\chi}G\bigl((\chi^{3})^{m}\bigr)G\bigl((\chi^{3})^{-1}\bigr)\chi(a)
  = \frac1N \cdot N\sum_{x \in \Kx}\sum_{\substack{y\in\Kx\\ y^{3}=ax^{3m}}}
      \psi(x+y)
\]
by Lemma~\ref{lem:gaussprod}, i.e.\ exactly the double sum in the statement.
The second contributes
\[
  \frac{Q}{N}\Bigl(1-\frac1Q\Bigr)\sum_{\chi^{3}=1}\chi(a)
   = \frac{Q-1}{N}\sum_{\chi^{3}=1}\chi(a)
   = \sum_{\chi^{3}=1}\chi(a),
\]
because $N = Q-1$.  It remains to identify $\sum_{\chi^{3}=1}\chi(a)$ with the
number of cube roots of $a$ in $\Kx$.  This is the standard count: by the two
orthogonality relations,
\[
  N\,\bigl|\{b : b^{3}=a\}\bigr|
   = \sum_{b \in \Kx}\sum_{\chi}\chi\bigl(b^{3}a^{-1}\bigr)
   = \sum_{\chi}\Bigl(\sum_{b\in\Kx}\chi^{3}(b)\Bigr)\chi(a^{-1})
   = N\sum_{\chi^{3}=1}\chi(a^{-1}),
\]
since $\sum_{b}\chi^{3}(b)$ is $N$ or $0$ according as $\chi^{3} = 1$ or not.
Dividing by $N$ and using that $\chi \mapsto \chi^{-1}$ is an involution of the
set $\{\chi : \chi^{3}=1\}$ with $\chi^{-1}(a^{-1}) = \chi(a)$, we get
$\sum_{\chi^{3}=1}\chi(a) = |\{b \in \Kx: b^{3}=a\}|$.
\end{proof}

\begin{lemma}[Vanishing at non-cubes]\label{lem:Snoncube}
If $a \in \Kx$ is not a cube in $\Kx$, then $S(a) = 0$.  (This is vacuous when
$c=1$, by Lemma~\ref{lem:mu3basic}(2).)
\end{lemma}
\begin{proof}
Both terms of Lemma~\ref{lem:closedform} vanish.  The first is the number of
cube roots of $a$, which is $0$ by hypothesis.  For the second, fix
$x \in \Kx$ and suppose $y \in \Kx$ satisfies $y^{3} = a\,x^{3m}$.  Then
\[
  \bigl(y\,(x^{m})^{-1}\bigr)^{3} = y^{3}\,x^{-3m} = a,
\]
so $a$ would be a cube, a contradiction; hence every inner index set is empty.
\end{proof}

\subsection{The phase at cubes is the Fourier transform of
\texorpdfstring{$\Phi$}{Phi}}\label{ssec:transform}

\begin{lemma}[Transform identity]\label{lem:twoScube}
For every $A \in \Kx$,
\[
  2S(A^{3}) = \widehat\Phi(A) = \sum_{t \in K}\Phi(t)\,\psi(At).
\]
\end{lemma}
\begin{proof}
Apply Lemma~\ref{lem:closedform} with $a = A^{3}$.

\emph{First term.}  By Lemma~\ref{lem:mu3basic}(1) the number of cube roots of
$A^{3}$ in $\Kx$ is exactly $c$.

\emph{Second term.}  Fix $x \in \Kx$.  The set
$\{y \in \Kx : y^{3} = A^{3}x^{3m}\}$ is the fibre of the cubing map through
$Ax^{m}$, hence equals $\{u\,A\,x^{m} : u \in \muthree\}$, again by
Lemma~\ref{lem:mu3basic}(1); the parametrisation $u \mapsto uAx^{m}$ is
injective.  Therefore
\[
  \sum_{\substack{y\in\Kx\\ y^{3}=A^{3}x^{3m}}}\psi(x+y)
   = \sum_{u \in \muthree}\psi\bigl(x + uAx^{m}\bigr),
\]
and, interchanging the two finite sums,
\[
  2S(A^{3}) = c + \sum_{u \in \muthree}\ \sum_{x \in \Kx}
     \psi\bigl(x+uAx^{m}\bigr).
\]

\emph{Restoring the point $x=0$.}  Since $m \ge 1$, the summand at $x = 0$ is
$\psi(0) = 1$ for each $u$, so
$\sum_{x\in K} = 1 + \sum_{x \in \Kx}$ termwise in $u$, and the $c$ ones
produced this way absorb the leading constant $c$:
\[
  2S(A^{3}) = \sum_{u \in \muthree}\ \sum_{x \in K}\psi\bigl(x+uAx^{m}\bigr).
\]

\emph{Substitution.}  Fix $u \in \muthree$ and put $\tau_{u} = (u^{-1})^{D}$,
which is nonzero.  The map $t \mapsto \tau_{u}t^{D}$ is a bijection of $K$,
being the composition of the bijection $t \mapsto t^{D}$
(Lemma~\ref{lem:powD}(1)) with multiplication by $\tau_{u} \neq 0$.  Writing
$x = \tau_{u}t^{D}$ we get, by Lemma~\ref{lem:powD}(1) again,
\[
  x^{m} = \bigl(\tau_{u}\bigr)^{m}\bigl(t^{D}\bigr)^{m}
        = u^{-1}\,t ,
  \qquad\text{hence}\qquad
  uAx^{m} = A\,t .
\]
Therefore
\[
  \sum_{x \in K}\psi\bigl(x+uAx^{m}\bigr)
   = \sum_{t \in K}\psi\bigl(\tau_{u}t^{D} + At\bigr).
\]

\emph{Reindexing $\muthree$.}  By Lemma~\ref{lem:mu3basic}(3),
$u \mapsto \tau_{u} = (u^{-1})^{D}$ permutes $\muthree$, so
\[
  2S(A^{3}) = \sum_{v \in \muthree}\sum_{t \in K}\psi\bigl(vt^{D}+At\bigr)
   = \sum_{t \in K}\Bigl(\sum_{v \in \muthree}\psi\bigl(vt^{D}\bigr)\Bigr)\psi(At)
   = \sum_{t\in K}\Phi(t)\psi(At) = \widehat\Phi(A),
\]
where the middle equality is $\psi(\alpha+\beta) = \psi(\alpha)\psi(\beta)$ and
an interchange of two finite sums.
\end{proof}

\begin{lemma}[Reindexing the triple sum]\label{lem:eightZ}
Let $A,B \in \Kx$.  Then
\[
  8\,Z(A^{3})\cdot c
   = \sum_{z \in \Kx}\widehat\Phi(z)\,\widehat\Phi(Az)\,\widehat\Phi(Bz),
\]
where $Z(A^{3})$ is the phase triple sum at $(\rho,\sigma) = (A^{3},B^{3})$.
\end{lemma}
\begin{proof}
Put $f(\lambda) = 8\,S(\lambda)S(A^{3}\lambda)S(B^{3}\lambda)$ for
$\lambda \in \Kx$, so that $8Z(A^{3}) = \sum_{\lambda \in \Kx}f(\lambda)$ by
the definition of the phase triple sum.

\emph{Values at cubes.}  For $z \in \Kx$ we have
$A^{3}z^{3} = (Az)^{3}$ and $B^{3}z^{3} = (Bz)^{3}$, so by
Lemma~\ref{lem:twoScube} applied at the three units $z$, $Az$, $Bz$,
\[
  f(z^{3}) = \bigl(2S(z^{3})\bigr)\bigl(2S((Az)^{3})\bigr)\bigl(2S((Bz)^{3})\bigr)
     = \widehat\Phi(z)\widehat\Phi(Az)\widehat\Phi(Bz).
\]

\emph{Fibrewise counting.}  Sorting the sum $\sum_{z \in \Kx}f(z^{3})$ by the
value $\lambda = z^{3}$,
\[
  \sum_{z \in \Kx}f(z^{3})
   = \sum_{\lambda \in \Kx}\bigl|\{z \in \Kx : z^{3}=\lambda\}\bigr|\,f(\lambda).
\]
For each $\lambda$ we claim
$|\{z : z^{3}=\lambda\}|\,f(\lambda) = c\,f(\lambda)$.  Indeed, if $\lambda$ is
a cube the cardinality is $c$ by Lemma~\ref{lem:mu3basic}(1); and if $\lambda$
is not a cube then $S(\lambda) = 0$ by Lemma~\ref{lem:Snoncube}, so
$f(\lambda) = 0$ and both sides vanish.  Hence
\[
  \sum_{z\in\Kx}\widehat\Phi(z)\widehat\Phi(Az)\widehat\Phi(Bz)
   = \sum_{z\in\Kx}f(z^{3})
   = c\sum_{\lambda\in\Kx}f(\lambda)
   = 8\,Z(A^{3})\,c . \qedhere
\]
\end{proof}

\subsection{Parseval and symmetrisation}\label{ssec:parseval}

\begin{lemma}[Vanishing mass]\label{lem:phihatzero}
$\sum_{t \in K}\Phi(t) = 0$, and consequently $\widehat\Phi(0) = 0$.
\end{lemma}
\begin{proof}
Interchanging the two sums and using Lemma~\ref{lem:powD}(2) for each
$u \in \muthree$, which is nonzero,
\[
  \sum_{t\in K}\Phi(t) = \sum_{u \in \muthree}\sum_{t \in K}\psi\bigl(ut^{D}\bigr)
   = \sum_{u \in \muthree}0 = 0 .
\]
Since $\psi(0\cdot t) = 1$, this is precisely $\widehat\Phi(0)$.
\end{proof}

\begin{lemma}[Parseval]\label{lem:parseval}
For all $A,B \in K$,
\[
  \sum_{z\in K}\widehat\Phi(z)\widehat\Phi(Az)\widehat\Phi(Bz)
   = Q \sum_{x,y \in K}\Phi(x)\Phi(y)\Phi(Ax+By).
\]
\end{lemma}
\begin{proof}
Expanding the three transforms and multiplying out,
\[
  \widehat\Phi(z)\widehat\Phi(Az)\widehat\Phi(Bz)
   = \sum_{x,y,w \in K}\Phi(x)\Phi(y)\Phi(w)\,
      \psi(zx)\psi(Azy)\psi(Bzw),
\]
and $\psi(zx)\psi(Azy)\psi(Bzw) = \psi\bigl(z(x+Ay+Bw)\bigr)$.  Summing over
$z \in K$ and interchanging sums, the orthogonality relation
$\sum_{z \in K}\psi(zs) = Q$ for $s = 0$ and $0$ for $s \neq 0$ (valid because
$\psi$ is primitive; this is the case $D = 1$ of the computation in
Lemma~\ref{lem:powD}(2)) gives
\[
  \sum_{z\in K}\widehat\Phi(z)\widehat\Phi(Az)\widehat\Phi(Bz)
  = Q\sum_{\substack{x,y,w\in K\\x+Ay+Bw=0}}\Phi(x)\Phi(y)\Phi(w).
\]
In characteristic two the constraint is $x=Ay+Bw$, so the last sum is
$\sum_{y,w}\Phi(Ay+Bw)\Phi(y)\Phi(w)$.  Renaming $(y,w)$ as $(x,y)$ and
reordering the factors proves the assertion.
\end{proof}

\begin{lemma}[Symmetrisation]\label{lem:symmetrisation}
For all $A,B \in K$,
\[
  \sum_{x,y \in K}\Phi(x)\Phi(y)\Phi(Ax+By)
   = c \sum_{u,v \in \muthree} W_{u,v}(A,B).
\]
\end{lemma}
\begin{proof}
For $a,b,g \in \muthree$ put
\[
  V(a,b,g) = \sum_{x,y \in K}
     \psi\bigl(a x^{D} + b y^{D} + g (Ax+By)^{D}\bigr),
\]
so that, expanding the three copies of $\Phi$,
\begin{equation}\label{eq:VVV}
  \sum_{x,y\in K}\Phi(x)\Phi(y)\Phi(Ax+By)
   = \sum_{a,b,g \in \muthree} V(a,b,g).
\end{equation}
Note that $V(u,1,v) = W_{u,v}(A,B)$.  Fix $b \in \muthree$ and set
$\tau = (b^{-1})^{m} \neq 0$, so $\tau^{D}=b^{-1}$ by Lemma~\ref{lem:powD}(1).
The substitution $(x,y)\mapsto(\tau x,\tau y)$ is bijective and gives
\[
  V(a,b,g) = V(ab^{-1},1,gb^{-1}).
\]
For fixed $b$, multiplication by $b^{-1}$ permutes $\muthree$ in both the
first and third coordinates.  Thus
\[
  \sum_{a,g\in\muthree}V(a,b,g)
   = \sum_{u,v\in\muthree}W_{u,v}(A,B),
\]
independently of $b$.  Summing over the $c$ choices of $b$ and using
\eqref{eq:VVV} proves the claim.
\end{proof}

\subsection{Weil sums are root counts}\label{ssec:weil}

We recall the definitions of $W_{u,v}(A,B)$ and $R_{u,v}(A,B)$ from
Definition~\ref{def:rootcount}:
\begin{align*}
  W_{u,v}(A,B) &= \sum_{x,y\in K}\psi\bigl(ux^{D}+v(Ax+By)^{D}+y^{D}\bigr),\\
  R_{u,v}(A,B) &= \bigl|\{t\in K:ut^{D}+v(At+B)^{D}=1\}\bigr|.
\end{align*}

\begin{lemma}[Weil sum evaluation]\label{lem:weil}
Let $u,v \in \muthree$ and $A,B \in K$ with $A^{3}\neq 1$.  Then
\[
  W_{u,v}(A,B)=Q\bigl(R_{u,v}(A,B)-1\bigr).
\]
\end{lemma}
\begin{proof}
If $u+vA^{D}=0$, cubing gives $(A^{3})^{D}=1$, hence $A^{3}=1$ by
Lemma~\ref{lem:powD}(1), a contradiction.  Thus the sum over $x$ in the slice
$y=0$ vanishes by Lemma~\ref{lem:powD}(2).  For $y\neq0$, substitute $x=ty$;
then
\[
  ux^{D}+v(Ax+By)^{D}+y^{D}
   = C(t)y^{D},\qquad C(t)=ut^{D}+v(At+B)^{D}+1.
\]
For fixed $t$, the sum over $y\neq0$ is $Q-1$ if $C(t)=0$ and $-1$ otherwise,
again by Lemma~\ref{lem:powD}(2).  Summing over $t$ gives
$W_{u,v}(A,B)=Q|\{t:C(t)=0\}|-Q$.  Since the characteristic is two,
$C(t)=0$ is equivalent to $ut^{D}+v(At+B)^{D}=1$, proving the formula.
\end{proof}

\subsection{Proof of Theorem~\ref{thm:phaseroot}}\label{ssec:assemblyroot}

\begin{proof}[Proof of Theorem~\ref{thm:phaseroot}]
First, $B \neq 0$: if $B = 0$ then $A^{3} = A^{3}+B^{3} = 1$, contradicting the
hypothesis $A^{3} \neq 1$.  (So the hypothesis $B \neq 0$ need not be
postulated separately.)  Thus $A,B \in \Kx$ and Lemma~\ref{lem:eightZ} applies.

By Lemma~\ref{lem:phihatzero} the summand of
$\sum_{z\in\Kx}\widehat\Phi(z)\widehat\Phi(Az)\widehat\Phi(Bz)$ vanishes at
$z = 0$ --- indeed already its first factor does --- so the sum over $\Kx$ may
be replaced by the sum over $K$.  Chaining
Lemmas~\ref{lem:eightZ}, \ref{lem:parseval}, \ref{lem:symmetrisation}
and~\ref{lem:weil} in this order,
\begin{align*}
  8\,Z(A^{3})\,c
   &= \sum_{z \in \Kx}\widehat\Phi(z)\widehat\Phi(Az)\widehat\Phi(Bz)
    = \sum_{z \in K}\widehat\Phi(z)\widehat\Phi(Az)\widehat\Phi(Bz)\\
   &= Q\sum_{x,y \in K}\Phi(x)\Phi(y)\Phi(Ax+By)
    = Q\,c\sum_{u,v \in \muthree}W_{u,v}(A,B)\\
   &= Q\,c\sum_{u,v\in\muthree}Q\bigl(R_{u,v}(A,B)-1\bigr)
    = c\,Q^{2}\Bigl(\sum_{u,v\in\muthree}R_{u,v}(A,B) - c^{2}\Bigr),
\end{align*}
the last step because the index set $\muthree \times \muthree$ has $c^{2}$
elements.  Since $c \ge 1$ we may cancel the factor $c$, and dividing by $8$,
\[
  Z(A^{3}) = \frac{Q^{2}}{8}
    \Bigl(\sum_{u,v \in \muthree}R_{u,v}(A,B) - c^{2}\Bigr),
\]
as asserted.  Note that the hypothesis $A^{3}+B^{3} = 1$ entered only through
$B \neq 0$; it is the normalisation under which the pair $(\rho,\sigma) =
(A^{3},B^{3})$ satisfies $\sigma = 1+\rho$, i.e.\ under which the left-hand
side is the phase triple sum of Definition~\ref{def:slopes}.
\end{proof}

Two consequences of Theorem~\ref{thm:phaseroot} drive the rest of the proof.

\begin{corollary}[Cube-support collapse]\label{cor:noncube}
Assume the Walsh formula of Theorem~\ref{thm:walshchar}.  If $\rho$ or
$1+\rho$ is not a cube in $K$, then $Z(\rho) = 0$.
\end{corollary}
\begin{proof}
If $c = 1$ every element is a cube and there is nothing to prove; so assume
$c=3$, i.e.\ $n$ even, and let (say) $\rho$ be a non-cube.  The cubes form a
subgroup of index $3$ of $\Kx$, so for each $\lambda \in \Kx$ at least one of
$\lambda$, $\rho\lambda$ is a non-cube; by Lemma~\ref{lem:Snoncube},
$S$ vanishes at non-cubes, so every summand of
$Z(\rho) = \sum_{\lambda}S(\lambda)S(\rho\lambda)S(\sigma\lambda)$ vanishes.
The case where $1+\rho$ is a non-cube is symmetric.
\end{proof}

\begin{corollary}[Positivity from root existence]\label{cor:positivity}
Assume the Walsh formula, let $\rho = A^{3}$ be an admissible cube slope with
$A, B \in K$ and $A \neq 0$ such that $A^{3}+B^{3} = 1$, and suppose that the
following holds:
\begin{equation}\label{eq:solvable}
  \text{for all } p,q \in \Kx \text{ with } p^{3}+q^{3}=1
  \text{ there exist } w,z \in K \text{ with } w+z=1,\; w^{m}+p z^{m} = q .
\end{equation}
Then $\ReX Z(\rho) \ge 0$.
\end{corollary}
\begin{proof}
By Theorem~\ref{thm:phaseroot} it suffices to show
$R_{u,v}(A,B)\ge 1$ for all $u,v \in \muthree$.  Put $a = u^{m}$, $b = v^{m}$;
then $a^{3} = b^{3} = 1$, and $a,b \neq 0$.  Apply \eqref{eq:solvable} to the
target
\[
  p = \frac{bA}{a}, \qquad q = bB,
\]
which is legitimate: $p,q \neq 0$ and
$p^{3}+q^{3} = A^{3}+B^{3} = 1$ because $a^{3}=b^{3}=1$.  We obtain $w,z$ with
$w+z=1$ and $w^{m} = pz^{m}+q$ (recall $-1 = 1$).  Set
\[
  t = \frac{z^{m}}{a} .
\]
Then $t^{D} = \bigl((z/u)^{m}\bigr)^{D} = z/u$, so $u t^{D} = z$.  Moreover
\[
  At+B = \frac{Az^{m}}{a}+B = \frac{1}{b}\bigl(p z^{m} + q\bigr)
       = \frac{w^{m}}{b} = \Bigl(\frac{w}{v}\Bigr)^{m},
\]
whence $(At+B)^{D} = w/v$ and $v(At+B)^{D} = w$.  Therefore
$u t^{D}+v(At+B)^{D} = z+w = 1$, i.e.\ $t$ is a root and
$R_{u,v}(A,B)\ge 1$.  Consequently
$\sum_{u,v}R_{u,v}\ge c^{2}$ and $Z(\rho) = \tfrac{Q^{2}}{8}(\sum R_{u,v}-c^{2})$
is a nonnegative real number.
\end{proof}

Condition \eqref{eq:solvable} is the only geometric input of the whole proof.
We call it $\mathrm{RootEqSolvable}(m,K)$ and establish it in the next section.

\section{The Fermat cubic and the root equation}\label{sec:geometry}

\begin{theorem}[Root existence]\label{thm:rootexistence}
Let $|K| = 2^{n}$, $\gcd(k,n)=1$ and $2^{k}+1 = 3m$, and in the even case also
$\gcd(m,2^{n}-1)=1$.  Then $\mathrm{RootEqSolvable}(m,K)$ holds: for all
$p,q \in \Kx$ with $p^{3}+q^{3}=1$ there are $w,z \in K$ with $w+z=1$ and
$w^{m}+pz^{m}=q$.
\end{theorem}

The two cases, $n$ even and $n$ odd, use the same geometric setup and differ in how a preimage under
$1+\pi^{k}$ is produced.

\subsection{The curve, its charts, and the addition law}

Let $E$ be the projective Fermat cubic $X^{3}+Y^{3}+Z^{3}=0$ over $K$ with
origin $O = (1:1:0)$. For the addition on $E(K)$, we use the notation $P_1\boxplus P_2$ for the sum of points $P_1$ and $P_2$, $\boxminus P$ for the inverse of a point $P$, $P_1\boxminus P_2 := P_1\boxplus(\boxminus P_2)$ for their difference, and $[j]P$ for the $j$-fold sum of a point $P$. It is immediate that the points at infinity of $E$ are $(1:a:0)$, where $a^3=1$. The inflection points of $E$ are $(1:a:0)$, $(0:1:a)$, $(a:0:1)$ with $a^3=1$; these form the 3-torsion subgroup $E[3]$. In particular, the number of $K$-rational $3$-torsion points is $3$ or $9$, depending on whether $n$ is odd or even.

The projectivised Hessian formula allows us to describe the group law in a simple form; see \cite[p. 425]{Chudnovsky}. In projective coordinates, for two points
\[
 P_1=(X_1:Y_1:Z_1),\qquad P_2=(X_2:Y_2:Z_2),
\]
on $E$, we define
\begin{align*}
 X_3&=Y_1^2X_2Z_2+Y_2^2X_1Z_1,\\
 Y_3&=X_1^2Y_2Z_2+X_2^2Y_1Z_1,\\
 Z_3&=X_1Y_1Z_2^2+X_2Y_2Z_1^2.
\end{align*}
Whenever $(X_3,Y_3,Z_3)\ne(0,0,0)$, the formula returns
$P_1\boxplus P_2=(X_3:Y_3:Z_3)$.  The exceptional locus is exactly the set of
pairs whose difference is a point at infinity.  Indeed, if all three displayed
coordinates vanish and one of $X_1,Y_1,Z_1$ is zero, the Fermat equation together
with those identities forces the corresponding coordinate of $P_2$ to vanish
as well, and the difference is one of the points $(1:a:0)$ with $a^3=1$.  If
all coordinates are nonzero, after scaling to $Z_1=Z_2=1$ the vanishing
identities give $Y_2=Y_1a$ and $X_2=X_1a^{-1}$ for some $a^3=1$, again
equivalent to $P_1\boxminus P_2=(1:a:0)$.  The converse follows by substituting
the translation formula below.
Inverse and doubling are particularly simple in characteristic $2$:
\[
\boxminus (X:Y:Z) = (Y:X:Z),\qquad [2](X:Y:Z) = (Y^4: X^4: Z^4).
\]
We summarise these formulas in the following proposition for affine points on the Fermat cubic.

\begin{proposition}[Affine Hessian addition laws]\label{prop:hessian}
For two affine Fermat points $(w_{1},t_{1})$, $(w_{2},t_{2})$ define the
Hessian quantities
\[
  D_{0} = w_{1}t_{1}+w_{2}t_{2},
  \qquad
  N_{x} = t_{1}^{2}w_{2}+t_{2}^{2}w_{1},
  \qquad
  N_{y} = w_{1}^{2}t_{2}+w_{2}^{2}t_{1}.
\]
\begin{enumerate}
  \item The quantities $D_0$, $N_x$, and $N_y$ vanish simultaneously if and only if $(w_1,t_1) \boxminus (w_2,t_2)$ is a point at infinity.
  \item An affine Fermat point is $3$-torsion if and only if one of its coordinates vanishes. 
  \item If $D_{0}\neq 0$ then
  \[
    (w_{1},t_{1}) \boxplus (w_{2},t_{2})
    = \Bigl(\frac{N_{x}}{D_{0}},\ \frac{N_{y}}{D_{0}}\Bigr)
  \]
  in the group $E(K)$.  
  \item Inversion and doubling are given by
  \[
    \boxminus (w,t) = (t,w),\qquad [2](w,t) = (t^4, w^4).
  \]
  \item Translation by a point at infinity $P_{a} = (1:a:0)$ is
  \[
    (w,t)\boxplus P_{a} = (aw,\ a^{2}t) \qquad (a^{3}=1),
  \]
  and $P_{a} \in E[3]$ for every $a$ with $a^{3}=1$.  
\end{enumerate}
\end{proposition}
\begin{proof}
All assertions follow by direct computation using the formulas above with affine coordinates $(w_i,t_i)=(X_i/Z_i,Y_i/Z_i)$ of $P_i$ for $i=1,2$. 
\end{proof}

Let $\pi$ denote the Frobenius $x \mapsto x^{2}$, both on $K$ and as a group endomorphism of $E$; in Fermat coordinates $\pi(w,t) = (w^{2},t^{2})$, and $\pi(P_{a}) = P_{a^{2}}$. For an integer $\ell$, $1+\pi^\ell$ denotes the endomorphism $P \mapsto P \boxplus \pi^\ell(P)$.
Three structural facts about $\pi$ are used repeatedly:

\begin{proposition}\label{prop:frobfacts}
\begin{enumerate}
\item (CM relation) $\pi^{2} = [-2]$ on $E(L)$ for every field $L \supseteq
      \F_{2}$ of characteristic two.
\item $\ker(1+\pi)$ consists exactly of the three points at infinity
      $P_{a}$, $a^{3}=1$, over any field of characteristic two.
\item If $|K| = 2^{n}$ then $\pi^{n} = \mathrm{id}$ on $E(K)$.
\end{enumerate}
\end{proposition}
\begin{proof}
(1) follows from Proposition~\ref{prop:hessian}(4) for affine Fermat points. Points at infinity are 3-torsion and defined over $\mathbb{F}_4$, hence fixed by $[-2]$ and $\pi^2$.  (2) Working with homogeneous coordinates,
$(1+\pi)(X:Y:Z) = O$ means $\pi(X:Y:Z) = \boxminus(X:Y:Z) = (Y:X:Z)$, i.e.\ there is
$c \neq 0$ with $X^{2} = cY$, $Y^{2}=cX$, $Z^{2}=cZ$.  If $Z \neq 0$ then
$c = Z$, so $X^{4} = Z^{2}Y^{2} = Z^{3}X$, i.e.\ the affine coordinates
$x = X/Z$ and $y = Y/Z$ satisfy $x^{4}=x$ and $y^{4}=y$, hence lie in $\F_{4}$
and have cubes in $\{0,1\}$.  The affine equation $x^{3}+y^{3}=1$ then forces
$x = 0$ or $y = 0$; but $x = 0$ gives $X = 0$ and then $ZY = X^{2} = 0$, so
$Y = 0$ as well, contradicting $x^{3}+y^{3}=1$, and symmetrically for $y=0$.
Hence $Z=0$, and conversely every point at infinity is in the kernel.  (3) Every coordinate of a $K$-rational point
satisfies $x^{2^{n}} = x$.
\end{proof}

\subsection{The incidence chart and the twisted root equation}

 Let $T_{3} = (1,0)$, an affine Fermat point that is $3$-torsion.  For an affine
 Fermat point
$\mathcal{Q} = (W,T)$ define
\[
  \Phi_{k}(\mathcal{Q}) = \boxminus\bigl(\mathcal{Q} \boxplus \pi^{k}\mathcal{Q}\bigr) \boxplus T_{3}.
\]
Here $\pi^{k}\mathcal{Q} = (W^{2^{k}}, T^{2^{k}})$.

\begin{theorem}[Twisted root equation]\label{thm:twisted}
Let $W,T$ be nonzero elements in the algebraic closure of $K$ with $W^{3}+T^{3}=1$, let $\gcd(k,n)=1$, suppose
\[
  \pi^{n}\mathcal{Q} = \mathcal{Q} \boxplus P_{\alpha} \quad\text{for some } \alpha^{3}=1
  \qquad (\mathcal{Q} = (W,T)),
\]
and suppose $\Phi_{k}(\mathcal{Q})$ is not $3$-torsion.  Then $\Phi_{k}(\mathcal{Q})$ is an affine
Fermat point $(p,q)$ with $p,q \neq 0$ and
\[
  W^{2^{k}+1} + p\,T^{2^{k}+1} = q .
\]
\end{theorem}

\begin{corollary}\label{cor:reduction-to-phi}
Suppose that for every affine Fermat point $(p,q)$ of $E$ with $p,q\neq 0$
there exists an affine $\mathcal{Q} = (W,T)$, defined over some extension of $K$ but with
$W^{3}, T^{3} \in K$, such that $W,T \neq 0$, the Frobenius twist relation
$\pi^{n}\mathcal{Q} = \mathcal{Q}\boxplus P_{\alpha}$ holds for some $\alpha^{3}=1$, and
$\Phi_{k}(\mathcal{Q}) = (p,q)$.  Then $\mathrm{RootEqSolvable}(m,K)$ holds.
\end{corollary}

\begin{proof}
Given a target $(p,q)$, take $\mathcal{Q}=(W,T)$ as in the hypothesis; $\Phi_{k}(\mathcal{Q}) =
(p,q)$ is not $3$-torsion since $p,q\neq 0$, so Theorem~\ref{thm:twisted}
gives $W^{2^{k}+1}+pT^{2^{k}+1} = q$.  Put $w = W^{3}$, $z = T^{3}$.  Then
$w+z = W^{3}+T^{3} = 1$, and since $3m = 2^{k}+1$,
\[
  w^{m} + p z^{m} = W^{3m}+pT^{3m} = W^{2^{k}+1}+pT^{2^{k}+1} = q .
\]
Both $w$ and $z$ lie in $K$ by hypothesis, so the pair $(w,z)$ witnesses
solvability.
\end{proof}

The proof of Theorem \ref{thm:twisted} occupies the remainder of this subsection.  Throughout it we fix an affine Fermat point $\mathcal{Q} = (W,T)$ with $W,T \neq 0$ and abbreviate
\[
  A = W^{2^{k}}, \qquad B = T^{2^{k}},
\]
so that $\pi^{k}\mathcal{Q} = (A,B)$; raising $W^{3}+T^{3}=1$ to the power $2^{k}$ and
using freshman's dream shows that $A^{3}+B^{3} = 1$, i.e.\ $(A,B)$ is again an
affine Fermat point.  We write
\[
  D_{0} = WT+AB, \qquad
  N_{x} = T^{2}A + B^{2}W, \qquad
  N_{y} = W^{2}B + A^{2}T
\]
for the Hessian data of the ordered pair $\mathcal{Q}$, $\pi^{k}\mathcal{Q}$.  Note that
$A \neq 0$ and $B \neq 0$.  All computations take place in a field of
characteristic two, so signs are irrelevant and $x+x = 0$ throughout.


\begin{lemma}[Weighted sum of the Hessian data]\label{lem:tw-weighted}
For every affine Fermat point $(W,T)$ and every $k$,
\begin{equation}\label{eq:weighted}
  W^{2^{k}+1}\,N_{y} \;+\; N_{x}\,T^{2^{k}+1} \;=\; D_{0}.
\end{equation}
\end{lemma}
\begin{proof}
Since $W^{2^{k}+1} = AW$ and $T^{2^{k}+1} = BT$, the left-hand side equals
\[
  AW\bigl(W^{2}B+A^{2}T\bigr) + \bigl(T^{2}A+B^{2}W\bigr)BT
   = ABW^{3} + A^{3}WT + ABT^{3} + B^{3}WT .
\]
Grouping the first with the third and the second with the fourth term,
\[
  AB\bigl(W^{3}+T^{3}\bigr) + WT\bigl(A^{3}+B^{3}\bigr) = AB + WT = D_{0},
\]
by the two Fermat equations.
\end{proof}

Identity \eqref{eq:weighted} is the whole point of the construction: after
division by $N_{y}$ it \emph{is} the twisted root equation.  The rest of the
proof consists in showing that the division is legitimate, i.e.\ that
$D_{0}$, $N_{x}$, and $N_{y}$ are all nonzero, and that the resulting pair
$(N_{x}/N_{y},\,D_{0}/N_{y})$ is precisely $\Phi_{k}(\mathcal{Q})$.


\begin{lemma}[Unfolding the twist relation]\label{lem:tw-unfold}
Let $\alpha^{3}=1$.  Then $\pi^{n}\mathcal{Q} = \mathcal{Q}\boxplus P_{\alpha}$ if and only if
\[
  W^{2^{n}} = \alpha W \quad\text{and}\quad T^{2^{n}} = \alpha^{2}T .
\]
In that case $(WT)^{2^{n}}=WT$, $\bigl(W^{3}\bigr)^{2^{n}} = W^{3}$, and
$\bigl(T^{3}\bigr)^{2^{n}} = T^{3}$.
\end{lemma}
\begin{proof}
By Proposition~\ref{prop:hessian}, $\mathcal{Q}\boxplus P_{\alpha} = (\alpha W, \alpha^{2}T)$,
while $\pi^{n}\mathcal{Q} = (W^{2^{n}}, T^{2^{n}})$; two affine Fermat points are equal
exactly when their coordinates agree, whence the stated equivalence.  Cubing
the two coordinate identities and using $\alpha^{3}=1$ gives
$(W^{3})^{2^{n}} = (W^{2^{n}})^{3} = \alpha^{3}W^{3} = W^{3}$ and
$(T^{3})^{2^{n}} = \alpha^{6}T^{3} = T^{3}$.
\end{proof}


\begin{lemma}[The Hessian data never vanish simultaneously]\label{lem:tw-nodegen}
Let $W,T \neq 0$, $\gcd(k,n)=1$, and assume the twist relation
$\pi^{n}\mathcal{Q} = \mathcal{Q}\boxplus P_{\alpha}$ for some $\alpha^{3}=1$.  Then
$D_{0} = N_{x} = N_{y} = 0$ is impossible.
\end{lemma}
\begin{proof}
  Assume for contradiction that 
  \begin{align*}
    D_{0} &= WT + AB = WT + W^{2^k}T^{2^k} = 0, \\
    N_{x} &= T^{2}A + B^{2}W = T^2W^{2^k} + T^{2^{k+1}}W = 0, \\
    N_{y} &= W^{2}B + A^{2}T = W^2T^{2^k} + W^{2^k}T^2 = 0.
  \end{align*}
  Applying Lemma \ref{lem:tw-unfold} and $\gcd(n,k)=1$, we obtain $WT=(WT)^2$. Since $WT \neq 0$, this implies $WT = 1$. But then
  \[ 
    N_{x} = W^{2^k-2} + W^{-2^{k+1}+1} = 0,
  \]
  equivalently, $W^{3(2^k-1)}=1$. Combining this with $(W^3)^{2^n-1}=1$, we get $W^3=1$. Then $T^3=1+W^3=0$, a contradiction. 
\end{proof}


\begin{lemma}[Non-vanishing]\label{lem:tw-nonvanishing}
Under the hypotheses of Theorem~\ref{thm:twisted} we have
\[
  D_{0} \neq 0, \qquad N_{x} \neq 0, \qquad N_{y} \neq 0 .
\]
\end{lemma}
\begin{proof}
By Lemma \ref{lem:tw-nodegen}, $D_{0} = N_{x} = N_{y} = 0$ is impossible. By Proposition~\ref{prop:hessian}, these are the homogeneous coordinates of the affine Fermat point $\mathcal{Q}\boxplus \pi^{k}\mathcal{Q}$. If one of them were zero, then $\mathcal{Q}\boxplus \pi^{k}\mathcal{Q}$ would be $3$-torsion, and then
\[
  [3](\Phi_{k}(\mathcal{Q})) = [3](\boxminus(\mathcal{Q}\boxplus\pi^{k}\mathcal{Q})\boxplus T_3) = [-3](\mathcal{Q}\boxplus\pi^{k}\mathcal{Q}) \boxplus [3]T_3 = O,
\]
contradicting the assumption that $\Phi_{k}(\mathcal{Q})$ is not $3$-torsion.
\end{proof}


\begin{proof}[Proof of Theorem~\ref{thm:twisted}]
The projective coordinates of the sum $\mathcal{Q}\boxplus \pi^{k}\mathcal{Q}$ are given by $(N_{x}:N_{y}:D_{0})$, and all three coordinates are nonzero by Lemma~\ref{lem:tw-nonvanishing}. We can use the Hessian addition laws to compute $\Phi_{k}(\mathcal{Q}) = \boxminus(\mathcal{Q}\boxplus \pi^{k}\mathcal{Q})\boxplus T_{3}$ in these coordinates.
Indeed, $\boxminus(N_{x}:N_{y}:D_{0}) = (N_{y}:N_{x}:D_{0})$, and adding
$T_{3} = (1:0:1)$ by the projective formulas gives
\[
  \Phi_{k}(\mathcal{Q}) = (N_{x}^{2} : N_{x}D_{0} : N_{x}N_{y}) = (N_{x}:D_{0}:N_{y}),
\]
using $N_{x} \neq 0$.
This means that $\Phi_k(\mathcal{Q})$ is an affine point with coordinates $(p,q)$, where 
\[
  p = \frac{N_{x}}{N_{y}} \neq 0, \qquad q = \frac{D_{0}}{N_{y}} \neq 0.
\]
Divide the weighted identity \eqref{eq:weighted} of Lemma~\ref{lem:tw-weighted} by $N_{y} \neq 0$:
\[
  W^{2^{k}+1} + \frac{N_{x}}{N_{y}}\,T^{2^{k}+1} = \frac{D_{0}}{N_{y}},
  \qquad\text{i.e.}\qquad
  W^{2^{k}+1} + p\,T^{2^{k}+1} = q . \qedhere
\]
\end{proof}

\begin{remark}
It is worth recording where each hypothesis is spent. The assumption that
$\Phi_{k}(\mathcal{Q})$ is not $3$-torsion is what rules out the two degenerate
configurations, $D_{0}=0$ (the two points sum to a point at infinity) and
$N_{x}N_{y}=0$ (the sum is an affine $3$-torsion point).  The twist relation
$\pi^{n}\mathcal{Q} = \mathcal{Q}\boxplus P_{\alpha}$ together with $\gcd(k,n)=1$ and $W,T \neq 0$ is
used exactly once, in Lemma~\ref{lem:tw-nodegen}, to exclude the totally
degenerate case in which all three Hessian quantities vanish.  In that case
$W^{3}$ and $T^{3}$ are fixed by $\pi^{k}$, and the twist relation makes them
fixed by $\pi^{n}$ as well; it is only the coprimality of $k$ and $n$ that then
pushes them into the prime field and produces the contradiction.
\end{remark}

\subsection{Odd dimension: the rational kernel is trivial}

\begin{proposition}\label{prop:oddkernel}
Let $|K| = 2^{n}$ with $n$ odd and $\gcd(k,n)=1$.  Then a point of $E(K)$ lies
in $\ker(1+\pi^{k})$ if and only if it is the origin.  Consequently
$1+\pi^{k}$ is a bijection of the finite group $E(K)$, and hence so is the map
$\Phi_k(P)=\boxminus((1+\pi^k)P)\boxplus T_3$.
\end{proposition}
\begin{proof}
As in Proposition~\ref{prop:frobfacts}(2), the kernel condition on a
homogeneous triple $(X:Y:Z)$ reads: there is $c\neq 0$ with
$X^{2^{k}} = cY$, $Y^{2^{k}} = cX$, $Z^{2^{k}} = cZ$.  If $Z \neq 0$ then
$c = Z^{2^{k}-1}$ and one computes that the affine coordinate $x = X/Z$
satisfies $x^{2^{2k}} = x$; since also $x^{2^{n}} = x$ and
$\gcd(2k,n) = \gcd(2,n) = 1$ for odd $n$, Bézout gives $x^{2} = x$, so
$x \in \{0,1\}$.  The kernel equations also give $y=x^{2^{k}}=x$, so
$x^{3}+y^{3}=0$, contradicting the affine Fermat equation $x^{3}+y^{3}=1$.
$Z = 0$ and the point is at infinity, $(1:\omega:0)$ with $\omega^{3}=1$;
 the kernel condition becomes $\omega^{2^{k}} = \omega^{2}$.  For odd $n$ the
 only cube root of unity
in $K$ is $\omega = 1$ (as $\gcd(3,2^{n}-1)=1$), giving the origin.  Injectivity of $1+\pi^{k}$ on the finite abelian group
$E(K)$ follows, hence bijectivity, and $\Phi_{k}$ is the composition of
$1+\pi^{k}$ with a negation and a translation.
\end{proof}

\begin{proof}[Proof of Theorem~\ref{thm:rootexistence}, odd $n$]
Let $(p,q)$ be a target, $p,q \neq 0$, $p^{3}+q^{3}=1$; it is an affine Fermat
point $P$ of $E(K)$, not $3$-torsion by
Proposition~\ref{prop:hessian}.  By Proposition~\ref{prop:oddkernel} there is a
unique $\mathcal{Q} \in E(K)$ with $\Phi_{k}(\mathcal{Q}) = P$.  $\mathcal{Q}$ is not $3$-torsion:
$\Phi_{k}$ maps $E[3]$ into $E[3]$, whereas $P$ is not $3$-torsion.  In
particular $\mathcal{Q}$ is not a point at infinity, so it is an affine Fermat
point $(W,T)$, and $W,T \neq 0$ by the torsion criterion of
Proposition~\ref{prop:hessian}.  Since $\mathcal{Q}$ is $K$-rational,
$\pi^{n}\mathcal{Q} = \mathcal{Q} = \mathcal{Q}\boxplus P_{1}$, so the Frobenius twist hypothesis holds with
$\alpha = 1$, and $W^{3},T^{3}\in K$ trivially.  Corollary
\ref{cor:reduction-to-phi} concludes.
\end{proof}

\subsection{Even dimension: invert the prime-to-three factor, then lift}

For even $n$ the rational kernel of $1+\pi^{k}$ is exactly the set of the three
points at infinity, so $\Phi_{k}$ is not
injective on $E(K)$ and the odd argument breaks down.  Instead we factor the
isogeny.

\begin{lemma}[{Factorisation in $\Z[\pi]$}]\label{lem:factor}
Let $k = 2r+1$ be odd with $2^{k}+1 = 3m$.  There are integers $a,b$ with
\[
  a+2b = 1, \qquad a-b = (-2)^{r}, \qquad a^{2}+2b^{2} = m ;
\]
explicitly $b = (1-(-2)^{r})/3$ and $a = (1+2(-2)^{r})/3$.  Set
$G = a - b\pi$ and $H = a+b\pi$, acting on any abelian group equipped with an
endomorphism $\pi$ satisfying $\pi^{2} = [-2]$.  Then
\[
  (1+\pi)\circ G = 1 + \pi^{k},
  \qquad
  H \circ G = G \circ H = [m] .
\]
If moreover some integer $M$ annihilates the group and $\gcd(m,M)=1$, then $G$
is bijective.
\end{lemma}
\begin{proof}
 The integrality of $a,b$ follows from $(-2)^{r}\equiv 1 \pmod 3$; the three displayed
identities are then direct computations, using $a^{2}+2b^{2} =
\bigl((a+2b)^{2}+2(a-b)^{2}\bigr)/3 = (1+2\cdot 4^{r})/3 = (2^{k}+1)/3 = m$.
For the factorisation, $\pi^{2}=[-2]$ gives $\pi^{2r} = [(-2)^{r}]$, so
$1+\pi^{k} = 1+(-2)^{r}\pi = (1+\pi)(a-b\pi)$ once $a+2b=1$ and
$a-b = (-2)^{r}$.  The norm identity gives $HG = [a^{2}+2b^{2}] = [m]$.
Finally, if $M$ annihilates the group and $\gcd(m,M)=1$, Bézout provides
integers $s,t$ with $sm+tM = 1$, so $[m]$ is invertible (with inverse $[s]$)
and $G$ is bijective with inverse $[s]H$.
\end{proof}

\begin{lemma}[An annihilator in even dimension]\label{lem:annihilator}
Let $|K| = 2^{n}$ with $n$ even, and put $c_{n} = (-2)^{n/2}$ and
$\mathcal N = c_{n}-1$.  Then $c_{n}^{2} = 2^{n}$, $\mathcal N \mid 2^{n}-1$,
and $\mathcal N$ annihilates $E(K)$.  If $\gcd(m,2^{n}-1)=1$ then
$\gcd(\mathcal N, m) = 1$.
\end{lemma}
\begin{proof}
$c_{n}^{2} = (-2)^{n} = 2^{n}$ since $n$ is even, and
$2^{n}-1 = (c_{n}-1)(c_{n}+1)$, giving the divisibility.  On $E(K)$,
$\pi^{n} = \mathrm{id}$ by Proposition~\ref{prop:frobfacts}(3), while
$\pi^{n} = (\pi^{2})^{n/2} = [(-2)^{n/2}] = [c_{n}]$ by the CM relation; hence
$[\mathcal N] = [c_{n}]-\mathrm{id} = 0$ on $E(K)$.  Coprimality is inherited
from $\mathcal N \mid 2^{n}-1$.
\end{proof}

\begin{corollary}\label{cor:Gbij}
Let $n$ be even, $k$ odd, $|K| = 2^{n}$, $2^{k}+1 = 3m$, and
$\gcd(m,2^{n}-1)=1$.  Then $G = a-b\pi$ is a bijection of $E(K)$; in
particular every $R \in E(K)$ has a $K$-rational preimage $Y$ with
$G(Y) = R$.
\end{corollary}
\begin{proof}
Combine Lemmas~\ref{lem:factor} and~\ref{lem:annihilator}.
\end{proof}

\begin{lemma}[Explicit quotient coordinates]\label{lem:quotient}
Let $L$ be a field of characteristic two and let $(x,y)$ be an affine point of
$x^{3}+y^{3}=1$ with $x,y\neq 0$.  Let $W \in L$ satisfy
$W^{3} = \dfrac{x+1}{x+y}$ and put
\[
  T = \frac{1}{W}\cdot\frac{1+x+y}{x+y}.
\]
Then
\[
  T^{3} = \frac{y+1}{x+y},
  \qquad W^{3}+T^{3} = 1,
\]
and the Hessian data of the pair $(W,T)$, $(W^{2},T^{2})$ satisfy $D_{0}\neq 0$
and
\[
  \frac{N_{x}}{D_{0}} = x, \qquad \frac{N_{y}}{D_{0}} = y ,
\]
i.e.\ $(W,T)\boxplus \pi(W,T) = (x,y)$.
\end{lemma}
\begin{proof}
All statements are rational identities in $x,y$ once $W^{3}$ is substituted,
using $x+y\neq 0$, $1+x+y\neq 0$ (both consequences of the cubic equation and
$x,y \neq 0$), and the key identity $1+(x+y)^{3} = (x+y)xy$.  For $T^{3}$:
\[
  T^{3} = \frac{(1+x+y)^{3}}{(x+y)^{3}}\cdot\frac{x+y}{x+1}
        = \frac{(1+x+y)^{3}}{(x+y)^{2}(x+1)} = \frac{y+1}{x+y},
\]
the last step being the polynomial identity
$(1+x+y)^{3} = (y+1)(x+1)(x+y)$ modulo $x^{3}+y^{3}+1$ in characteristic two.
Then $W^{3}+T^{3} = \frac{(x+1)+(y+1)}{x+y} = 1$.  For the Hessian data put
$s = WT = (1+x+y)/(x+y)$, so $1+s = 1/(x+y)$.  Then
$D_{0} = WT + W^{2}T^{2} = s(1+s) \neq 0$,
$N_{x} = T^{2}W^{2} + T^{4}W = WT^{2}(W+T^{2})$, and
$N_{y} = W^{2}T^{2} + W^{4}T = W^{2}T(T+W^{2})$, whence
\[
  \frac{N_{x}}{D_{0}} = \frac{s + T^{3}}{1+s} = \frac{x/(x+y)}{1/(x+y)} = x,
  \qquad
  \frac{N_{y}}{D_{0}} = \frac{s + W^{3}}{1+s} = \frac{y/(x+y)}{1/(x+y)} = y .
\]
\end{proof}

The quantities $W^{3}$, $T^{3}$, and $WT$ in Lemma~\ref{lem:quotient} are
rational functions of $(x,y)$ alone: they do \emph{not} depend on the choice of
the cube root $W$.  This is exactly the statement that they are the invariants
of the degree-three Hessian quotient, whose kernel is the group of the three
points at infinity; cf.~\cite{TwistedHessian,HessianIsogenies}.  It is what
makes the descent in the next proof free of charge.

\begin{proof}[Proof of Theorem~\ref{thm:rootexistence}, even $n$]
Let $(p,q)$ be a target, and let $P = (p,q) \in E(K)$ be the corresponding
affine Fermat point; $P$ is not $3$-torsion.

\emph{(1) Invert $G$ over $K$.}  Put $R = T_{3}\boxminus P \in E(K)$.  By
Corollary~\ref{cor:Gbij} there is $Y \in E(K)$ with $G(Y) = R$.

\emph{(2) $Y$ is affine with nonzero coordinates.}  $R$ is not $3$-torsion,
since $T_{3}$ is and $P$ is not; and $G$ maps $3$-torsion to $3$-torsion, so
$Y$ is not $3$-torsion either.  Hence $Y = (x,y)$ with $x,y \neq 0$
by Proposition~\ref{prop:hessian}.

\emph{(3) Lift through the degree-three quotient.}  Let $\overline{K}$ be an algebraic
closure of $K$ and regard $(x,y)$ in $E(\overline{K})$ via the base-change homomorphism
$E(K)\to E(\overline{K})$.  Choose $W \in \overline{K}$ with
$W^{3} = (x+1)/(x+y)$ and let $T$ be as in Lemma~\ref{lem:quotient}; then
$\mathcal{Q} = (W,T) \in E(\overline{K})$ satisfies $\mathcal{Q}\boxplus \pi \mathcal{Q} = (x,y)$.

 \emph{(4) Reassemble the factorisation.}  Base change commutes with $\pi$,
 and $G$ is a polynomial in $\pi$, hence commutes with
$1+\pi$ and with base change.  Therefore, in $E(\overline{K})$,
\[
  (1+\pi)\bigl(G(\mathcal{Q})\bigr) = G\bigl(\mathcal{Q}\boxplus \pi \mathcal{Q}\bigr) = G\bigl((x,y)\bigr)
    = R = T_{3}\boxminus P,
\]
and by Lemma~\ref{lem:factor}, $(1+\pi^{k})\mathcal{Q} = (1+\pi)G(\mathcal{Q}) = T_{3}\boxminus P$, i.e.
\[
  \Phi_{k}(\mathcal{Q}) = \boxminus(1+\pi^{k})\mathcal{Q} \boxplus T_{3} = P .
\]

\emph{(5) Control the Frobenius defect.}  $\mathcal{Q}$ itself need not be
$K$-rational, but $\mathcal{Q}\boxplus \pi \mathcal{Q} = (x,y)$ is, hence is fixed by $\pi^{n}$.  Since
$\pi^{n}$ commutes with $1+\pi$,
\[
  (1+\pi)\bigl(\pi^{n}\mathcal{Q} \boxminus \mathcal{Q}\bigr) = \pi^{n}(\mathcal{Q}\boxplus \pi \mathcal{Q})\boxminus(\mathcal{Q}\boxplus \pi \mathcal{Q}) = O,
\]
so $\pi^{n}\mathcal{Q}\boxminus \mathcal{Q} \in \ker(1+\pi)$, which by
Proposition~\ref{prop:frobfacts}(2) is the set of the three points at infinity.
Thus $\pi^{n}\mathcal{Q} = \mathcal{Q} \boxplus P_{\alpha}$ for some $\alpha$ with $\alpha^{3}=1$: the
Frobenius twist hypothesis of Theorem~\ref{thm:twisted} holds.  Also
$W,T \neq 0$, since $\Phi_{k}(\mathcal{Q}) = P$ is not $3$-torsion.

\emph{(6) Descend the invariant cubes.}  Theorem~\ref{thm:twisted}, applied
over $\overline{K}$, gives $W^{2^{k}+1}+pT^{2^{k}+1} = q$.  Setting $w = W^{3}$,
$z = T^{3}$ we get $w+z=1$ and $w^{m}+pz^{m} = q$ as in
Corollary~\ref{cor:reduction-to-phi}.  Finally $w$ and $z$ lie in $K$: by
Lemma~\ref{lem:quotient} they equal $(x+1)/(x+y)$ and $(y+1)/(x+y)$
respectively, which are rational expressions in the $K$-rational coordinates
$x,y$.  No fixed-field descent theorem is needed --- the cube root $W$ was
chosen only \emph{after} the $K$-rational data $w,z$ had been written down.
\end{proof}

\section{Assembling the proof}\label{sec:assembly}

\begin{lemma}[Forcing]\label{lem:forcing}
Let $s$ be a nonempty finite index set, $Q>0$, $c \in \R$, and let
$N_{i} = c + Z_{i}/Q$ with $Z_{i}\ge 0$ for all $i \in s$.  If
$\frac{1}{|s|}\sum_{i\in s} N_{i} = c$, then $N_{i} = c$ and $Z_{i}=0$ for
every $i \in s$.
\end{lemma}
\begin{proof}
Summing the hypothesis, $\sum_{i}N_{i} = |s|c + \bigl(\sum_{i}Z_{i}\bigr)/Q$,
while the average condition gives $\sum_{i}N_{i} = |s|c$.  Hence
$\sum_{i}Z_{i}=0$ with all $Z_{i}\ge 0$, so each $Z_{i}=0$ and $N_{i}=c$.
\end{proof}

\begin{theorem}[Normalised slope form]\label{thm:slopeform}
Let $|K| = 2^{n}$, $1\le k<n$, $k$ odd, $\gcd(k,n)=1$, $2^{k}+1 = 3m$,
$\gcd(m,2^{n}-1)=1$, and $2|\Delta_{k}| = Q$.  Then
\[
  N(\rho) = \frac{Q^{2}}{8} \qquad \text{for every admissible } \rho .
\]
\end{theorem}
\begin{proof}
Let $D$ be an inverse of $m$ modulo $N = |\Kx|$; 
note $m \neq 0$ and $\mathcal S \neq \emptyset$ because
$n \ge 2$.  Theorem~\ref{thm:phase}
gives the phase formula, and Theorem~\ref{thm:walshchar} converts it into the
all-character Walsh formula for $S$, the Walsh transform of $\Delta_{k}$.

\emph{Pointwise lower bound.}  Let $\rho$ be admissible.  If $\rho$ or
$1+\rho$ is not a cube, then $Z(\rho) = 0$ by Corollary~\ref{cor:noncube}.
Otherwise write $\rho = A^{3}$, $1+\rho = B^{3}$ with $A,B \in K$; then
$A \neq 0$ (as $\rho \neq 0$), $A^{3}+B^{3} = \rho + 1 + \rho = 1$, and
$A^{3} = \rho \neq 1$.  Theorem~\ref{thm:rootexistence} supplies
$\mathrm{RootEqSolvable}(m,K)$ --- through the odd branch if $n$ is odd and
through the even branch if $n$ is even --- so Corollary~\ref{cor:positivity}
gives $\ReX Z(\rho)\ge 0$.

\emph{Conclusion.}  By Theorem~\ref{thm:walshcount},
$N(\rho) = Q^{2}/8 + \ReX Z(\rho)/Q$ for every admissible $\rho$, and by
Theorem~\ref{thm:average} the average of $N$ over $\mathcal S$ is exactly
$Q^{2}/8$.  Lemma~\ref{lem:forcing}, with $c = Q^{2}/8$ and $Z_{\rho} =
\ReX Z(\rho)$, yields $N(\rho) = Q^{2}/8$ for every admissible $\rho$ (and, as a
by-product, $Z(\rho) = 0$ throughout: the Fourier correction vanishes
identically).
\end{proof}

\begin{proof}[Proof of Theorem~\ref{thm:main}]
Let $K$, $n$, $k$, $v_{1}$, $v_{2}$ be as in the statement.  By
Proposition~\ref{prop:reduction}, $n \ge 2$ and we may replace $k$ by
$r = k\bmod n$, which satisfies $1\le r<n$ and $\gcd(r,n)=1$, without changing
$N_{k}(v_{1},v_{2})$.

By Result~\ref{res:halfsize}, $2|\Delta_{r}| = Q$.  By
Proposition~\ref{prop:normparam} choose $k_{0}\in\{r,n-r\}$ odd, coprime to
$n$, with $2^{k_{0}}+1 = 3m$ and $\gcd(m,2^{n}-1)=1$; note $1 \le k_{0}<n$.  By
Proposition~\ref{prop:transport} (used only if $k_{0} = n-r$) there are
distinct nonzero $w_{1},w_{2}$ with
$N_{r}(v_{1},v_{2}) = N_{k_{0}}(w_{1},w_{2})$, and by
Result~\ref{res:halfsize} again $2|\Delta_{k_{0}}| = Q$.

Theorem~\ref{thm:slopeform} applies at the parameter $k_{0}$ and gives
$N(\rho) = Q^{2}/8$ for every admissible slope.  By Lemma~\ref{lem:coefslope},
$N_{k_{0}}(w_{1},w_{2}) = Q^{2}/8$, and by \eqref{eq:mainterm} this equals
$2^{2n-3}$.  Chaining the equalities,
\[
  N_{k}(v_{1},v_{2}) = N_{r}(v_{1},v_{2}) = N_{k_{0}}(w_{1},w_{2}) = 2^{2n-3},
\]
which is the assertion.
\end{proof}

\begin{remark}[What the proof also shows]
The proof yields more than the count.  The forcing step shows that the Fourier
correction $Z(\rho)$ vanishes for every admissible slope, and, through
Theorem~\ref{thm:phaseroot}, that for even $n$ and every admissible slope
$\rho = A^{3}$ with $1+\rho = B^{3}$ the twisted root counts satisfy
$\sum_{u,v\in\muthree}R_{u,v}(A,B) = 9$ exactly, i.e.\ each twisted fibre
contains exactly one root; for odd $n$ the corresponding statement is
$R_{1,1}(A,B) = 1$.  The geometry alone gives ``at least one''; exactness comes
from the conserved average.
\end{remark}


\section*{Acknowledgements}
This research was supported by the National Research, Development and Innovation Fund of Hungary under grant SNN 152582. We gratefully acknowledge the support of Resoloon Technologies Ltd, which enabled the first author's access to Claude Fable~5. Special thanks go to \'Abris Nagy and to Claude Carlet for valuable comments and suggestions. 

\section*{Declaration of generative AI and AI-assisted technologies in the manuscript preparation process}
During the research and preparation of this work, the authors used several AI
systems under their guidance and through sustained mathematical discussions.
Claude Fable~5 (Anthropic) was used in the earlier stage of the project reported
in the Nagy--Vajda paper~\cite{NagyVajda}; Claude Opus~5 (Anthropic) and
ChatGPT~6 Pro (OpenAI) were used more extensively in developing and polishing
the present proof and manuscript.  The authors reviewed and edited the generated
material and take full responsibility for the mathematical claims, references,
and final content of the paper.  The results have been formalised in Lean~4 with
AI assistance, including Aristotle (Harmonic), and verified by the Lean kernel;
see \cite{palomar-2026-09-04-000005-v1}.

\printbibliography

\end{document}